\documentclass[12pt]{article}

\usepackage{amsfonts,amsmath,amsthm,amssymb,bm,bbm, thmtools, thm-restate}
\usepackage{latexsym,color,epsfig,mathrsfs,enumerate}
\usepackage[toc,page]{appendix}
\usepackage{fancyhdr}
\usepackage[colorlinks=true, allcolors=red]{hyperref}
\usepackage{indentfirst}
\usepackage[backend=biber, maxbibnames=9]{biblatex}
\usepackage{comment}
\usepackage{xcolor}
\usepackage[shortlabels]{enumitem}

\usepackage{graphicx}
\usepackage{pgfplots}
\pgfplotsset{compat=1.15}
\usetikzlibrary{arrows}

\nocite{*}

\usepackage{combelow}

\usepackage{newunicodechar}
\newunicodechar{ș}{\cb{s}}
\newunicodechar{ț}{\cb{t}}

\newtheorem{thm}{Theorem}[section]
\newtheorem{dfn}[thm]{Definition}
\newtheorem{lem}[thm]{Lemma}
\newtheorem{cor}[thm]{Corollary}

\newcommand{\laur}[1]{\textbf{\underline{#1}}}
\newcommand{\bad}[1]{\operatorname{bad}_{\mathcal{D}}({#1})}
\newcommand{\baddd}[4]{\operatorname{bad}_{\mathcal{#1}_{#2}}^{#3}({#4})}
\newcommand{\proj}[3]{\operatorname{proj}_{#1}({\bf {#2}}_{#3})}

\DeclareMathOperator{\diver}{div}

\usepackage{combelow}

\newcommand{\ov}{\overline}
\newcommand{\ul}{\underline}

\newcommand{\R}{\mathbb R}

\newcommand{\bP}{\mathbb P}
\newcommand{\bE}{\mathbb E}

\allowdisplaybreaks

\title{Distinct degrees and homogeneous sets II}
\author{Eoin Long\thanks{School of Mathematics, University of Birmingham, UK. Email: \texttt{e.long@bham.ac.uk}.\newline \hspace*{1.5em} 
The research of the first author is supported by an EPSRC New Investigator Award [EP/Y006399/1].} \and Lauren\cb{t}iu Ploscaru\thanks{School of Mathematics, University of Birmingham, UK. Email: \texttt{ixp090@student.bham.ac.uk}.\newline \hspace*{1.5em} 
The second author is grateful for support through an EPSRC DTP studentship.}}
\begin{document}

\maketitle

\begin{abstract}
    Given an $n$-vertex graph $G$, let $\hom (G)$ denote the size of a largest homogeneous set in $G$ and let $f(G)$ denote the maximal number of distinct degrees appearing in an induced subgraph of $G$. The relationship between these parameters has been well studied by several researchers over the last 40 years, beginning with Erd\H{o}s, Faudree and S\'os in the Ramsey regime when $\hom (G) = O(\log n)$.\vspace{2mm}
    
    Our main result here proves that any $n$-vertex graph $G$ with $\hom (G) \leq n^{1/2}$ satisfies
        \begin{align*}
            f(G) \geq \sqrt[3]{\frac {n^2}{\hom (G)} } \cdot n^{-o(1)}.
        \end{align*}
  This confirms a conjecture of the authors from a previous work, in which we addressed the $\hom (G) \geq n^{1/2}$ regime. Together, these provide the complete extremal relationship between these parameters (asymptotically), showing that any $n$-vertex graph $G$ satisfies
        \begin{align*}
            \max \Big ( f(G) \cdot \hom (G), \sqrt {f(G) ^3 \cdot \hom (G) } \Big ) \geq  n^{1-o(1)}.
        \end{align*}
  This relationship is tight (up to the $n^{-o(1)}$ term) for all possible values of $\hom (G)$, from $\Omega (\log n )$ to $n$, as demonstrated by appropriately generated Erd\H{o}s--Renyi random graphs.
    \end{abstract}

\section{Introduction }
\label{section: introduction}

 \indent The homogeneous number of an $n$-vertex graph $G$ is defined as
    \begin{align*}
        \hom (G) := \max _{U \subset V(G)} \big \{|U|: G[U] \mbox{ is a complete or empty graph} \big \}.
    \end{align*}
\indent This parameter is of central interest in Graph Theory and, in particular, in Ramsey Theory. One version of Ramsey's theorem \cite{ramsey1930}, which is a cornerstone result in Combinatorics, states that $\hom (G) \to \infty$ as $n \to \infty $. A more quantitative bound, due to Erd\H{o}s and Szekeres \cite{Erdosold}, gives that $\hom (G) \geq \log _2n /2$. Random graphs show that this logarithmic growth is best possible, as proven by Erd\H{o}s \cite{Erdosprob} in one of the earliest applications of the probabilistic method \cite{alon04}. However, no explicit constructions of $n$-vertex graphs with $\hom (G) = O(\log n)$, also called \emph{Ramsey graphs}, are known, despite carrying a $\$100$ Erd\H{o}s prize. There are several compelling results proving that graphs produced via certain natural constructive procedures cannot attain this level of behaviour, along with outstanding conjectures asserting even broader restrictions.

A guiding belief in Ramsey theory is that large graphs with small homogeneous number must resemble the behaviour of random graphs and, as such, on might view the homogeneous number of a graph as a measure of quasirandomness. However, while several other similar measures are known to be equivalent (see for example \cite{THOMASON-quasi}, \cite{quasi-2}, \cite{quasi}), this measurement is generally far weaker. Over time, Erd\H{o}s and his coauthors made several influential conjectures \cite{book-erdos},\linebreak 
asserting that $n$-vertex Ramsey graphs resemble the typical behaviour of the random graph $G(n,1/2)$. The vast majority of these questions have been affirmatively answered: Erd\H{o}s and Szemer\'edi \cite{erdosmrd} proved that Ramsey graphs have density bounded away from $0$ and $1$, Shelah \cite{Shelah} showed they must have exponentially many distinct induced subgraphs, R\"odl and Pr\"omel \cite{PROMEL} proved them to be $\Omega (\log n)$-universal, Kwan and Sudakov \cite{KSP}, \cite{kwan2018ramsey} confirmed they share similar subgraph statistics with $G(n,1/2)$, while recently Kwan, Sah, Sauermann and Sawhney \cite{KSSS} have shown that they have an induced subgraph of each size in the interval $[0, \Omega (n^2)]$. 

The focus of this paper is on the relationship between the homogeneous number of a graph $G$ and its distinct degree number. Given a graph $G$ we let
    \begin{align*}
        f(G) := \max \big \{ k\in {\mathbb N}: G[U] \mbox{ has } k \mbox{ vertices with distinct degrees, for some } U \subset V(G)\big \}.
    \end{align*}
\indent Understanding the relationship between $\hom (G)$ and $f(G)$ was first raised in the Ramsey regime by Erd\H{o}s, Faudree and S\'os \cite{sos}. They conjectured that every $n$-vertex graph $G$ with $\hom (G) = O(\log n)$ satisfies $f(G) = \Omega (n^{1/2})$, a lower bound which was easily observed to be true for $G(n,1/2)$ with high probability (whp). This conjecture was proved by Bukh and Sudakov in an influential paper \cite{bukh}. They also highlighted that that while $\Omega (n^{1/2})$ is a natural lower bound for $f(G(n, 1/2))$, there was no matching upper bound, and they proved instead that $f(G(n,1/2)) = O(n^{2/3})$ whp. Conlon, Morris, Samotij and Saxton \cite{unpublished} later provided a matching lower bound for the random graph, showing that $f(G(n, 1/2)) = \Theta (n^{2/3})$ whp. Jenssen, Keevash, Long and Yepremyan later resolved the Ramsey setting in \cite{JKLY}, proving that if $G$ is an $n$-vertex Ramsey graph then $f(G) = \Omega (n^{2/3})$. 

In their paper, Bukh and Sudakov raised the broader question of understanding $f(G)$ based on $\hom (G)$ when $\hom (G) \gg \log n$. In particular, they asked whether $\hom (G) = o(n)$ already guarantees that $f(G) \geq n^{1/2-o(1)}$. This was proven by Narayanan and Tomon \cite{narayanan}, who showed that in fact $f(G) \geq \Omega \big ( (n/\hom (G))^{1/2} \big )$ for all graphs $G$. In the same paper, they also conjectured that $f(G) \geq \Omega (n / \hom (G))$ provided $\hom (G) \geq n^{1/2}$, which is sharp, as demonstrated by the $\hom (G)$-partite Tur\'an graph. Later, Jenssen, Keevash, Long and Yepremyan \cite{JKLY} proved that the Tur\'an graph is indeed extremal for $f(G)$ when $\hom (G) \geq n^{9/10}$, improving an earlier bound of $\hom (G) \geq \Omega (n/\log n)$ given by Narayanan and Tomon. 

    In \cite{LongPL} the authors recently resolved the Narayanan-Tomon conjecture, up to a logarithmic loss, proving an essentially sharp result for $\hom (G) \geq n^{1/2}$. 

\begin{thm}
    \label{thm: DDandHS_prev_thm} There is an absolute constant $C>0$ such that the following holds true. Every $n$-vertex graph $G$ with $\hom (G) \geq n^{1/2}$ satisfies
        \begin{align*}
            f(G) 
                \geq 
            \frac {n}{\hom (G)} \cdot \log ^{-C}n.
        \end{align*}
\end{thm}

Theorem \ref{thm: DDandHS_prev_thm} thus gives a sharp dependency between these parameters when $\hom (G) \geq n^{1/2}$. However, as $\hom (G)$ descends beyond the above $n^{1/2}$ threshold, the lower bound here begins to fail $-$ this is illustrated by the random graph $G:=G(n,1/2)$, as $f(G) = O(n^{2/3})$ and $\hom (G) = O\big (\log n\big )$ whp. Bearing this example in mind, it is natural to investigate the behaviour of general random graphs $G(n, p)$ for $p \in [0,1]$. The authors gave a complete analysis in Section 6 of \cite{LongPL}, which we summarise below. Noting that both $\hom (G)$ and $f(G)$ are preserved under complement, it suffices to restrict to $p \leq 1/2$. Then observe:
    \begin{itemize}
        \item [(i)] $\hom (G(n, p)) = O(p^{-1} \log n) = p^{-1}n^{o(1)}$ whp, by a first moment argument;
        \item [(ii)] $f(G(n,p)) \leq \widetilde {O}(np)$ whp, as $f(G) \leq \Delta (G) +1$ for any graph $G$. 
        For $p \leq n^{-1/2}$ it is quite easy to observe that this bound is sharp, up to a $\log ^C n$ loss\footnote{This also follows from Theorem \ref{thm: DDandHS_prev_thm}.};
        \item [(iii)] A more subtle upper bound on $f(G(n,p))$ is obtained as follows. Let $H$ be an induced subgraph of $G(n,p)$ which has $d$ vertices of distinct degrees in $H$; then there is a set $S$ made of $d/4$ of these vertices and another vertex set $W$ disjoint to $S$ such that the event ${\cal E}_{S,W}$ that\linebreak
        $\big|e_{G(n,p)}(S, W) - p|S||W|\big| \geq d^2/16$ holds. If $d \leq pn/4$, by Chernoff's inequality we obtain:
        $${\mathbb P}({\cal E}_{S,W}) \leq \exp \big( - (d/4)^4 \cdot 4/d\cdot (4np)^{-1} \big) = \exp \big( - \Omega (d^3/pn) \big).$$ 
        When $d = \Omega \left( \sqrt[3]{n^2p} \right)$ above, we deduce that
            \begin{align*}
                {\mathbb P}\big ( f(G(n,p)) \geq d \big ) 
                        \leq 
                {\mathbb P}(\mbox {some }{\cal E}_{S,W} \mbox{ holds}) 
                        \leq 
                (2^n)^2 \cdot \exp \big( - \Omega (d^3/np) \big) \ll 1.
            \end{align*}
         This gives that $f\big (G(n,p)\big ) \leq \sqrt[3]{n^2p}\cdot n^{o(1)}$ whp for $p \in [n^{-1/2},1/2]$.
    \end{itemize}
    
      In \cite{LongPL} the authors proved that the final inequality in (iii) is tight for $p\in [n^{-1/2}, 1/2]$ which, combined with (i), shows that $G:=G(n,p)$ satisfies
    \begin{align}
        \label{eqn:lower-bound-for-f(G)-for-random}
        f(G) = \big ( {n^2}{p} \big )^{1/3} \cdot n^{o(1)} = \sqrt [3]{ \frac {n^2}{\hom (G)}} \cdot n^{o(1)} 
        \qquad \mbox{whp.}
    \end{align}        
We also conjectured the right-hand expression in \eqref{eqn:lower-bound-for-f(G)-for-random} provides a lower bound on $f(G)$ for all $n$-vertex graphs $G$ with $\hom (G) \leq n^{1/2}$, the complementary regime to Theorem \ref{thm: DDandHS_prev_thm}. 

Our main theorem here confirms this conjecture.

\begin{thm}
    \label{thm: DDandHS_thm} There is an absolute constant $C>0$ such that the following holds true. Every $n$-vertex graph $G$ with $\hom (G) \leq n^{1/2}$ satisfies
        \begin{align*}
            f(G) 
                \geq 
            \sqrt [3] {\frac {n^2}{\hom (G)}} \cdot e^{-C(\log n)^{2/3}}.
        \end{align*}
\end{thm}

We note that by combining Theorems \ref{thm: DDandHS_prev_thm} and \ref{thm: DDandHS_thm} this gives the following immediate corollary, which provides the complete extremal relationship between $\hom (G)$ and $f(G)$.

\begin{cor}\label{cor:complete-relationship}
    Given an $n$-vertex graph $G$, we have
        \begin{align*}
            \max \bigg ( f(G) \cdot \hom (G), \sqrt { f(G)^{3} \cdot \hom (G) }\bigg )
                \geq 
            n^{1-o(1)}.
        \end{align*}
\end{cor}

We highlight that Corollary \ref{cor:complete-relationship} is sharp in both regimes for appropriate random graphs. Thus the result could be interpreted as saying that random graphs determine the (asymptotic) extremal relation between $\hom (G)$ and $f(G)$. In particular, as mentioned above, the guiding belief that random graphs indicate the behaviour of graphs with small homogeneous numbers applies here throughout the entire regime, and appears to be one of the first known results of this character.\vspace{2mm} 

To illustrate another interesting instance of Theorem \ref{thm: DDandHS_thm}, we return to Narayanan and Tomon's result \cite{narayanan} that any $n$-vertex graph with $\hom (G) = n^{o(1)}$ satisfies $f(G) \geq n^{1/2-o(1)}$, which solved a question of Bukh and Sudakov \cite{bukh}. Theorem \ref{thm: DDandHS_thm} shows that this condition in fact produces essentially the same behaviour as was proven in \cite{JKLY} for Ramsey graphs.

\begin{cor}\label{cor:NT-BS-bound}
    Every $n$-vertex graph $G$ with $\hom (G) = n^{o(1)}$ satisfies $f(G) \geq n^{2/3-o(1)}$.
\end{cor}

Lastly, we highlight a result of a more general form, which may be of independent interest, but which will also be helpful when discussing our proof approach. Two vertices $u, v$ of a graph $G$ are said to have \emph{diversity} $D$ to a set $S\subset V(G)$ if $|N^S(u)\triangle N^S(v)| \geq D$. A subset $U \subset V(G)$ is said to be $D$-\emph{diverse} to $S$ if $|N^S(u_1) \triangle N^S(u_2)| \geq D$ for all distinct $u_1\neq  u_2 \in U$.

In \cite{bukh}, diversity was identified as a key parameter in analysing $f(G)$, and this connection has had a large impact on the study of several properties of Ramsey graphs (e.g. degree and size distributions of induced subgraphs \cite{kwan2018ramsey},\cite{narayanan},\cite{narayanan2017ramsey},\cite{JKLY}).
In particular, in \cite{bukh} Bukh and Sudakov proved that if $|U| = k$ and $D \geq k^{2}$ above, then $f(G) = \Omega (k)$. Using results from \cite{LongPL}, one can show that $D\geq k^{3/2}$ suffices to give this conclusion, which turns out to be essentially sharp. 

It is still natural however to suspect this result could be further improved if there was less \emph{pressure} on the vertices of $S$ here. For instance, at an extreme, if each vertex in $S$ only appears in one $N^S(u)$ with $u\in U$ then it is easy to see that $D \geq k$ already implies $f(G) = \Omega(k)$. Our next theorem provides a general bound, depending on the pressure vertices from $U$ to $S$.

\begin{thm}\label{thm:pressure-bound}
    Let $D \geq 1$ and $\gamma \in (0,1]$. Suppose $G$ is a graph with $U, S \subset V(G)$, such that:
        \begin{itemize}
            \item [(i)] $|N^S(u_1) \triangle N^S(u_2)| \geq D$ for all distinct $u_1, u_2 \in U$;
            \item [(ii)] $|N^U(s)| \leq \gamma |U|$ for all $s \in S$.
        \end{itemize}
    Then the following relation holds: $$f(G) = \Omega \bigg ( \frac {|U|}{\big (1 + \sqrt{\gamma |U|^{3}/D^2} \big )\cdot \log ^2|U|} \bigg ).$$
\end{thm}
\noindent Note in particular that Theorem \ref{thm:pressure-bound} gives $f(G) = |U|^{1-o(1)}$ if either: 
\begin{itemize}[nosep]
    \item $D = \Omega\big(|U|^{3/2}\big)$ and $\gamma = 1$ (trivially satisfied), or
    \item $D = \Omega (|U|)$ and $\gamma = |U|^{-1}$ -- see the example discussed above.
\end{itemize}

It is also easy to verify that whp the hypothesis of Theorem \ref{thm:pressure-bound} holds for the random graph $G(n,p)$ when $p \in [n^{-1/2},1/2]$, with $|U| = c(n^2p)^{1/3}$, $D = np/4$, $S = V(G)$ and $\gamma = 2p$. Therefore, we recover the essentially sharp bound $f(G(n,p)) \geq |U|^{1-o(1)} = (n^2p)^{1/3-o(1)}$ whp.\vspace{4mm} 

\noindent \textbf{Proof overview.} Before closing the introduction, we briefly discuss the proof of Theorem \ref{thm: DDandHS_thm}.\linebreak In \cite{LongPL} we introduced a general approach to establishing lower bounds on $f(G)$ for a graph $G$, which we used to prove Theorem \ref{thm: DDandHS_prev_thm}. 
Roughly speaking, instead of finding sets $U \subset W\subset V(G)$\linebreak such that $U$ has many distinct degrees in $G[W]$, it suffices instead to find such a set $U$ and a probability distribution ${\cal D}$ on $[0.1, 0.9]^{V(G)}$ such that certain small-ball like quantities associated with $U$ and ${\cal D}$ are well-controlled (see Section \ref{sec:machinery} for more details). This resulted in a more technical task, but with the benefit that it allowed us to blend behaviour from different neighbourhoods of the graph together to find distinct degrees, and that it was more robust to an inductive approach, which seems forced upon us in this type of problem.

Our proof of Theorem \ref{thm: DDandHS_thm} follows the same approach, although several substantial new ideas are required in the current regime. To begin, observe that if we could find sets $U$ and $S$ with $|U|$ large, which have diverse neighbourhoods and such that the pressure from $U$ to $S$ is small, then one could use Theorem \ref{thm:pressure-bound} to prove the theorem, as illustrated in the remark for $G(n,p)$ above. We thus restrict our attention to graphs $G$ in which one of the conditions (i) or (ii) from Theorem \ref{thm:pressure-bound} fails for such sets.

Should (i) fail then there are vertices in $G$ whose neighbourhoods cluster or significantly correlate. We investigate this correlation using the notion of `cluster neighbourhood', which allows us to separate vertices which are strongly correlated from others which have a smaller correlation. Under natural assumptions, we obtain a partial decomposition of a significant proportion of the graph using cluster neighbourhoods (see Lemma \ref{lem:cluster-partition-lemma}). We might now hope to find distinct degrees within each of these clusters by induction (as in \cite{LongPL}), and then combine them together to obtain $U$ using the weak correlation between the clusters in place of (i) in Theorem \ref{thm:pressure-bound}. 

Although this sounds plausible, it results in a new difficulty. Indeed, observe that this partition into clusters naturally complicates achieving (ii) from Theorem \ref{thm:pressure-bound}. Indeed, the `pressure parameter' $\gamma $ is necessarily forced to be much larger if $U$ has many vertices of a similar structure. To bypass this new restriction, we instead build our probability distributions ${\cal D}$ based on the coarse structure of the cluster neighbourhoods, rather through vertex neighbourhoods which were used previously. This enables us to work with a measure of pressure based on\linebreak the number of \emph{clusters}, rather than the sought number distinct degrees (as in Theorem \ref{thm:pressure-bound}). We make this connection by
 exploiting the metric structure of the neighbourhoods (Lemma \ref{lem: blended-distribution-control-sets}),\linebreak which will in turn be controlled through our definition of clusters.
 
 Lastly, with these tools in hand, we prove Theorem \ref{thm: DDandHS_thm} via a delicate inductive argument. In particular, we were only able to guarantee a low enough level of pressure for the cluster partition in certain circumstances (see Case 3 of the proof), but luckily a trivial level of pressure ($\gamma = 1$) turned out to be sufficient outside of this (see Case 2 of the proof).\vspace{3mm} 

\noindent \textbf{Structure of the paper.} In the next section we collect some classic tools from Graph Theory and Probabilistic Combinatorics, which will be useful in proving our theorem. In the beginning of Section \ref{sec:machinery} we recap the key elements of our approach from \cite{LongPL} to lower bounding $f(G)$. We furthermore provide a crucial lemma that allows us to find distributions which separate large parts of the graph with relatively low overall pressure. Next, in Section \ref{sect:cluster-neighbourhoods} we introduce cluster neighbourhoods and study their properties. In particular, we prove a central result which helps us find a partial decomposition of a graph into highly correlated pieces. Finally, in Section \ref{sec:proof} we combine these components together to prove Theorem \ref{thm: DDandHS_thm}.\vspace{3mm}

\noindent \textbf{Notation.} Given a graph $G$ and $u,v\in V(G)$, we write $u\sim v$ if $u$ and $v$ are adjacent vertices in $G$ and $u\not\sim v$ if they are not. The neighbourhood of $u$ is given by $N_G(u) = \{v\in V(G): u\sim v\}$ and given $S \subset V(G)$ we let $N^S_G(u) := N_G(u) \cap S$; we will omit the subscript $G$ when the graph is clear from the context. We write $d^S_G(u) = |N^S_G(u)|$. The maximum and the minimum degree of $G$, denoted by $\Delta(G)$ and $\delta(G)$ respectively, are the largest and the smallest possible degree of a vertex of $G$. The average degree of $G$, denoted usually by $\ov{d}(G)$, is simply the average of all vertex degrees of $G$, i.e. $\ov{d}(G):=|V(G)|^{-1}\cdot \sum_{v\in V(G)} d_G(v)$. 

\indent Given a vertex $u$, we will also represent the neighbourhood of $u$ by a vector $\textbf{u} \in \{0,1\}^{V(G)}$ defined such that $\textbf{u}_v=1$ if and only if $u\sim v$. Given a set $U \subset V$ and a vector ${\bf u} \in {\mathbb R}^V$, we will denote the projection of ${\bf u}$ onto the coordinate set $S$ by $\proj {S}{u}{}$, i.e. for any $v \in S$ we have $ \proj {S}{u}{} _v = {\bf u} _v$. Given $u,v\in V(G)$ we write $\text{div}_G(u,v)$ for the symmetric difference $N(u)\triangle N(v)$. Thus $|\text{div}_G(u,v)|$ is simply the Hamming distance between ${\bf u}$ and ${\bf v}$.

\indent We will write $\ov{G}$ for the complement of the graph $G$. It is easy to note that for any graph $G$ we have $\hom (G) = \hom (\ov{G})$ and $f(G)=f(\ov{G})$ since $\text{div}_G(u,v)=\text{div}_{\ov{G}}(u,v)$ for any $u,v\in V(G)$.

\indent Given $n\in \mathbb{N}$ and $p\in (0,1)$, the Erd\H{o}s$-$R\'enyi random graph $G(n,p)$ is the $n$-vertex graph in which each edge is included in the graph with probability $p$ independently of every other edge. We say that an event that depends on $n$ occurs \emph{with high probability} (whp) if its probability tends to $1$ as $n\to \infty$.

\indent Throughout this paper we will omit floor and ceiling signs when they are not crucial, for the sake of clarity of presentation.\vspace{4mm}

\section{Tools}

In this short section we introduce some tools required for the rest of the paper. 

\subsection{Graph theory tools}

 We will need the following version of Tur\'an's theorem (see for example Chapter 6 in \cite{bollo}).

\begin{thm}\label{turan}
Let $G$ be a $n$-vertex graph with average degree $d$. Then $G$ has an independent set of size at least $n/(d+1)$. 
\end{thm}

Besides this, the following result (Lemma 5.7 from \cite{LongPL}) will be quite advantageous, as it will enable us to move to a large induced subgraph that is reasonably regular. Comparable results were proved by Alon, Krivelevich and Sudakov in \cite{alon2008large}.

\begin{lem}\label{AKS-type-thm}
    Given a graph $G$ on $n$ vertices there is $A\subset V(G)$ 
    with $|A|\geq n/ 30\log n$ such that the induced subgraph $H = G[A]$ satisfies $\Delta (H) \leq 5\log n \cdot \delta (H)$.
\end{lem}

\subsection{Probabilistic Tools}



We will require some classic concentration inequalities. See e.g. appendix A in \cite{alon04}.

\begin{thm}[\textbf{Chernoff Inequality}]\label{cher}
Let $X$ be a random variable with binomial distribution and let $\mu = \bE[X]$. Then, for $0\leq \delta\leq 1$, the following inequalities hold:
$$ \bP\big(X\leq (1-\delta )\mu \big)\leq \exp\left(-{\dfrac {\delta ^{2}\mu }{2}}\right).$$
$$ \bP\big(X\geq (1+\delta )\mu \big)\leq \exp\left(-{\dfrac {\delta ^{2}\mu }{4}}\right).$$
\end{thm}

\begin{thm}[\textbf{Hoeffding's Inequality}]\label{hoeff}
Let $X_1,X_2,\dots, X_n$ be independent random variables such that $a_i\leq X_i\leq b_i$ for each $i\in[n]$, where $a_i,b_i\in \R$. Then given $t>0$, the random variable $S_n=X_1+\dots+X_n$ satisfies:
$$\bP\left(|S_n-\bE[S_n]|\geq t\right)\leq 2\cdot \emph{exp}\left(\dfrac{-2t^2}{\sum_{i\in[n]}(b_i-a_i)^2 }\right).$$
\end{thm}

The next result will be very useful when dealing with large deviations.

\begin{thm}\label{binomial-bound}
Let $n\in \mathbb{N},\ p\in [0,1],\ L>0$ an let $X\sim Bin(n,p)$ be a random variable. Then: $$\bP(X\geq L)\leq \binom{n}{L}p^L\leq \left(\frac{enp}{L}\right)^L.$$ 
\end{thm}

We end this section with a very useful lower bound $-$ due to Greenberg and Mohri \cite{quarter-binomial} $-$ for the probability that a binomial random variable exceeds is mean.

\begin{thm}\label{thm:quarter-binomial}
    Let $n\in \mathbb{N}$, $p\in(n^{-1},1)$ and let $X\sim \text{Bin}(n,p)$ be a random variable. Then: $$\bP(X\geq np)>\dfrac{1}{4}.$$
\end{thm}

\section{Distinct degrees through probability distributions}\label{sec:machinery}

As mentioned in the Introduction, in \cite{LongPL} we introduced a general framework which provided lower bounds on $f(G)$ for a graph $G$. Roughly speaking, in using this approach, one aims to find a set $U \subset V(G)$ together with a probability distribution ${\cal D}$ on $[0,1]^{V(G)}$ with the property that certain small-ball type quantities associated to $U$ are well controlled. We will summarise what we require here, but a more extensive presentation can be found in \cite{LongPL}, where the authors introduced this set up. In particular, see the conclusion of the introduction of \cite{LongPL} for some of the motivation for this approach along with Section 3 for further details and proofs.

Given a graph $G$ and a probability vector $\laur{p} = (p_v)_{v \in V(G)} \in [0.1, 0.9]^{V(G)}$ we will write $G(\laur{p})$ to denote the probability space on the set of induced subgraphs of $G$, determined by including each vertex $v \in V(G)$ independently with probability $p_v$. Equivalently, given $S \subset V(G)$, the induced subgraph $G[S]$ is selected with probability $\prod _{v \in S}p_v \prod _{v \in V(G) \setminus S} (1-p_v)$. Abusing notation slightly\footnote{As with the Erd\H{o}s--Renyi random graph $G(n,p)$.}, we will usually write $G(\laur{p})$ to denote a random graph $G[S] \sim G(\laur{p})$.

Throughout the paper, given a vertex $u \in V(G)$, we will we write $\textbf{u} \in \{0,1\}^{V(G)}$ to denote the \emph{neighbourhood vector} of $u$, which is given by 
    \begin{eqnarray*}
        ({\bf u})_v =
        \begin{cases} 
            1 \quad \mbox{ if } uv \in E(G);\\
            0 \quad \mbox{ otherwise}.
        \end{cases}
    \end{eqnarray*}
\indent Note that, considering the standard inner product on ${\mathbb R}^{V(G)}$, given by ${\bf x}\cdot {\bf y} = \sum _{v\in V(G)}x_vy_v$, this notation leads us to the following useful representation:
    \begin{eqnarray}
        \label{eqn: inner product expected degree}
        {\mathbb E}\big [ d_{G(\laur{p})} (u) \big ] = \bf u \cdot \laur{p} .
    \end{eqnarray}
Expected degrees in suitable $G(\laur{p})$'s prove to be useful in finding bounds for $f(G)$, however this notion is still too rigid for our purposes. The parameter we are about to define turns out to be robust enough.

 Let $G$ be a graph and let ${\cal D}$ be a probability distribution on $[0.1, 0.9]^{V(G)}$. Given distinct vertices $u,v \in V(G)$ and a set $S\subset V(G)$, we define
\begin{eqnarray}
\label{eqn:bad-definition}
\baddd{D}{}{S}{u,v}:=\max_{c\in \mathbb{R}}\ \underset{\laur{p}\sim \mathcal{D}}{\bP}\big(|\bE[d_{G(\laur{p})}^S(u)]-\bE[d_{G(\laur{p})}^S(v)]-c|\leq 1\big).
\end{eqnarray}
This quantity can be viewed as a small ball probability $-$ a measurement for two vertices $u,v \in V(G)$ of how likely the expected degrees to $S$  in $G(\laur{p})$ are to differ by an (almost) fixed amount. Given sets $U, S\subset V(G)$, we also set $$\baddd {\cal D}{}{S}{U}:=\displaystyle\sum_{\{u,v\}\subset U}\baddd {\cal D}{}{S}{u,v}.$$
Given another set $V \subset V(G)$ we can also write
$$\baddd {\cal D}{}{S}{U, V }:=\displaystyle\sum_{(u,v)\in U\times V}\baddd {\cal D}{}{S}{u,v}.$$
\indent We will sometimes suppress the superscript when $S = V(G)$, e.g. $\bad U = 
\baddd {\cal D}{}{V(G)}{U}$. Lastly, let us remark that in \eqref{eqn:bad-definition} we do not need $\mathcal{D}$ to be defined on all vertex coordinates of the set $[0.1,0.9]^{V(G)}$; any vertex set $T$ with $S\subseteq T\subseteq V(G)$ is enough so that we can define $\mathcal{D}$ on $[0.1,0.9]^T$, as we can see by looking at the RHS of \eqref{eqn:bad-definition}.    

As highlighted in \cite{LongPL}, in order to find many distinct degrees in a graph $G$ it suffices to find a large set $U \subset V(G)$ and a probability distribution ${\cal D}$ such that $\bad U$ is small. 

\begin{thm}
    \label{thm: bad-control-implies-distinct-expected-degrees}
    Let $G$ be a graph, let ${\cal D}$ be a probability distribution on $[0.1,0.9]^{V(G)}$ 
    and let $U \subset V(G)$ with $\bad U = \alpha \cdot |U|$. Then
    $$f(G) = \Omega \left( \frac {|U|}{ (1+\alpha)\cdot \log ^{3/2}|U| } \right).$$
\end{thm}

Thus producing a set $U$ and a distribution ${\cal D}$ with strong `bad' control results in a lower bound for $f(G)$. We also proved in \cite{LongPL} that one can reverse this process, to obtain a distribution with strong bad control from a lower bound on $f(G)$ giving that the two approaches are comparable, although we will not need this reverse direction here.

For the remainder of this section we will collect a number of results together, which will be used in combination to exhibit distributions ${\cal D}$ with bad control. As indicated earlier, the key benefit of working in this generalised setting is the flexibility of combining distributions on different sets while maintaining bad control. This is illustrated by the following result.

\begin{lem}
\label{lem: product-dist-lem}
Let $G$ be a graph with vertex partition $V(G)= \bigsqcup_{i=1}^L V_i$ and for each $i\in [L]$ let $\mathcal{D}_i$ be a probability distribution on $[0,1]^{V_i}$. Then taking $\mathcal{D}$ to denote the product distribution $\Pi_{i\in [L]} \mathcal{D}_i$ on $[0,1]^{V(G)}$, for any distinct vertices $u,v\in V(G)$ and any set $S\subset V(G)$, one has: $$\baddd{D}{}{S}{u,v}\leq \min_{i\in [L]} \baddd{D}{i}{S\cap V_i}{u,v}.$$
\end{lem}

\begin{proof}
See Lemma 4.1 in \cite{LongPL}.
\end{proof}

Our second lemma gives a simple situation in which we can obtain `bad' control. Let $G$ be a graph and let $S \subset V(G)$. Let ${\cal U}_S$ denote the \textbf{\emph {uniformly constant distribution}} on $[0.1,0.9]^S$, given by selecting $\alpha \in [0.1,0.9]$ uniformly at random and setting $\laur{p}
= \alpha {\bf 1}_S \in [0.1,0.9]^{S}$.

It is also convenient to write ${\cal T}_S$ to denote the \textbf{\emph{trivial $S$-induced probability distribution}}, the distribution on $[0.1,0.9]^S$ which simply selects the vector $\laur{p}_0 = \frac {1}{2} \cdot {\bf 1}_S$ with probability $1$. 

\begin{lem}
    \label{lem: uniform-distribution-control}
        Let $G$ be a graph, $S \subset V(G)$ and $u, v \in V(G)$ such that 
        $d^S(u) \geq d^S(v) + D$ for some $D >0$. Suppose that 
        ${\cal U}_S$ denotes the uniform constant distribution on $[0.1,0.9]^{S}$, that 
        ${\cal D}'$ denotes a distribution on $[0.1,0.9]^{V(G)\setminus S}$ and that 
        ${\cal D}$ denotes the product distribution 
        ${\cal U}_S \times {\cal D}'$ 
        on $[0.1,0.9]^{V(G)}$. Then $\bad {u,v} \leq 3D^{-1}.$
\end{lem}

\begin{proof}
   See Lemma 4.2 from \cite{LongPL}.
\end{proof}

We next seek to provide `bad' control for a set by blending neighbourhood structures together $-$ the idea here has some similarities to that of \cite{JKLY}. Let $G$ be a graph, let $U, S \subset V(G)$, where $U := \{u_1,\ldots, u_k\}$, and let $\beta\in[0,0.4]$. We now let ${\cal B}_\beta(U,S)$ denote the \textbf{\emph {blended probability distribution}} on $[0.1,0.9]^S$, which is defined as follows. First independently select $\alpha _i \in [-\beta,\beta]$ uniformly at random for each $i \in [k]$ and set: 
\begin{align}
    \label{eqn: pre truncated p}
\laur{p}' := \frac {1}{2} \cdot {\bf 1} + \sum _{i\in [k]} \alpha _i \cdot \proj {S}{u}{i} \in \mathbb{R}^{S}.
\end{align}
Having made these choices, the distribution then returns $\laur{p}$, a truncated version of $\laur{p}'$, where

    \begin{align*}
        \laur{p}_v =
        \begin{cases}
               \laur{p}_v' \quad \mbox{if } \laur{p}_v' \in [0.1,0.9];\\
               0.9 \quad \mbox{if } \laur{p}_v' >0.9;\\
               0.1 \quad \mbox{if } \laur{p}_v' < 0.1.
        \end{cases}
    \end{align*}

Our next lemma gives a useful extension of Lemma 4.3 from \cite{LongPL} which will allow us to obtain strong control on $\mbox{bad}_{\cal D}(v_i,v_j)$. In particular, it is designed to be used in conjunction with the partial decomposition from cluster neighbourhoods given in Section \ref{sect:cluster-neighbourhoods} (see Lemma \ref{lem:cluster-partition-lemma}).

\begin{dfn}
    Given a graph $G$, vertex sets $U,S\subset V(G)$ and a parameter $\gamma \in [0,1]$, we say $U$ is  $\gamma $-\emph{\textbf{balanced }}to $S$ if for all $v \in S$ we have $d_G^{U}(v) \leq \gamma |U|$. 
\end{dfn}

Observe that $U$ is always $1$-\emph{balanced to} $S$. 

\begin{lem}
      \label{lem: blended-distribution-control-sets}
        Let $G$ be a graph, $\beta\in(0,0.1),\ \gamma\in(0,1],\ D\geq 1$ and let $U:=\{u_1,u_2,\ldots,u_{t}\}$ and $S$ be subsets of $V(G)$ such that $U$ is $\gamma $-balanced to $S$. Suppose there are $d_1,d_2,\ldots, d_t>0$ and pairwise disjoint sets $V_1,V_2,\ldots, V_{t} \subset V(G)$ with $u_i\in V_i$ for all $i\in [t]$ such that:
        \begin{itemize} 
            \item $|\text{\emph{div}}^S_G(u_i,v_i)|\leq d_i$ for each $v_i\in V_i$ and $i\in [t]$;
            \item $|\text{\emph{div}}^S_G(u_i,u_j)|\geq D + d_i + d_j$ for all distinct $i, j \in [t]$.
        \end{itemize}
        Suppose further that ${\cal D}$ denote the product distribution ${\cal D} = {\cal B}_\beta(U,S) \times {\cal D}'$ on $[0.1,0.9]^{V(G)}$ where ${\cal B}_\beta(U,S)$ denotes the blended probability distribution on $[0.1,0.9]^{S}$ and ${\cal D}'$ is any distribution on $[0.1,0.9]^{V(G)\setminus S}$. Then for all $v_i\in V_i$ and $v_j\in V_j$ with $i\neq j$ in $[t]$ one has
                \begin{equation}\label{eqn:bad-eqn-weight-control}
                    \bad {v_i,v_j} 
                        \leq 
                    \dfrac{2}{\beta D} 
                    + 2 \cdot \max \big \{d^S_G(u_i),d^S_G(u_j) \big \}\cdot \exp\left(\dfrac{-0.045}{\gamma\beta^2|U|}\right).
                \end{equation}
\end{lem}

\begin{proof}
For each $i\in [t]$, given the vector $\laur{p}'$ on $\mathbb{R}^S$ from \eqref{eqn: pre truncated p}, define the random vector $\laur{q}^i$ on $\mathbb{R}^S$ by $\laur{q}^{i}:=\laur{p}^\prime-\alpha_i\cdot \proj {S}{u}{i}$. A key observation is that $\laur{q}^i$ is independent of $\alpha_{i}$. We will slightly abuse notation by writing $\laur{p}$ for both a vector in $[0.1,0.9]^{V(G)}$ and its projection $\proj{S}{\ul{p}}{}$ onto the coordinate set $S$. Since $\mathcal{D}$ is the product distribution ${\cal B}_\beta(R,S) \times {\cal D}'$, we can do this without much concern due to Lemma \ref{lem: product-dist-lem}. \\
    \indent Fix $i,j\in [t]$ with $i\neq j$ and $v_i\in V_i,\ v_j \in V_j$. Given $c\in \mathbb{R}$, let $E^{\ i,j}_{v_i,v_j}(c)$ denote the event that $\big|\bE[d_{G(\laur{p})}(v_i)]-\bE[d_{G(\laur{p})}(v_j)]-c\big|\leq 1$. According to \eqref{eqn:bad-definition}, it will suffice to show that: $$\bP\big(E^{\ i,j}_{v_i,v_j}(c)\big)\leq \dfrac{2}{\beta D}+2\max\{d^S_G(u_i),d^S_G(u_j)\}\cdot \exp\left(\dfrac{-0.045}{\gamma\beta^2|U|}\right).$$ 
    \indent Let us assume that $d_G^S(v_i)\geq d_G^S(v_j)$. Call a vertex $v\in N_G^S(u_i)$ \emph{naughty} if $\laur{q}_v^i\notin[0.2,0.8]$. We call the vertex $u_i$ \emph{problematic} if $N_G^S(u_i)$ contains a naughty vertex and we denote this event by $F_i$.\linebreak By the law of total probability we get that
    \begin{align}
    \label{eqn:prob-E-control}
    \bP\big(E^{\ i,j}_{v_i,v_j}(c)\big)=\bP\big(E^{\ i,j}_{v_i,v_j}(c)|F_{i}\big)\cdot \bP(F_{i})+\bP\big(E^{\ i,j}_{v_i,v_j}(c)|\ov{F_{i}}\big)\cdot \bP(\ov{F_{i}})\leq \bP(F_{i})+\bP\big(E^{\ i,j}_{v_i,v_j}(c)|\ov{F_{i}}\big).
    \end{align}
\indent To upper bound $\bP(F_i)$, let $v\in S$ and note that $\laur{q}_v^i$ is a sum of $ d_G^{U\setminus\{u_i\}}(v)$ uniform independent random variables, as the $v$-coordinate of $\textbf{u}_{i}$ is non-zero when $v\sim u_i$. Therefore, we have
$$\bP\big(\laur{q}_v^i\notin [0.2,0.8]\big)=\bP\big(|\laur{q}_v^i-1/2|>0.3 \big)
\leq 2\exp\left(\dfrac{-2\cdot 0.09}{4\beta^2 
d_G^{U\setminus \{u_i\} }(v) }\right) 
\leq 2\exp\left(\dfrac{-2\cdot 0.09}{4\beta^2 
\gamma |U| }\right),$$
where the first inequality follows from Hoeffding's inequality, whereas the second one uses that $d_G^{U\setminus \{u_i\} }(v) \leq d_G^U(v) \leq \gamma |U|$ as $U$ is 
$\gamma$-balanced to $S$. By the union bound we then get \begin{align}\label{eqn:F-control}
    \bP(F_{i})\leq d_G^S(u_i)\cdot \bP\big(\laur{q}_v^i\notin [0.2,0.8]\big)\leq 2d_G^S(u_i)\cdot \exp\big(-0.045(\gamma\beta^2|U|)^{-1}\big).
\end{align}

\indent To compute $\bP\big(E_{v_i,v_j}^{\ i,j}(c)|\ov{F_i}\big)$ we condition on any choice of $\bm{\alpha}:=(\alpha_{l})_{l\neq i}$ such that $F_i$ does not hold. Given such a choice, let us first note that $\laur{p}_v^\prime=\laur{q}^i_v+\alpha_i \textbf{u}_{i_v}\in [0.1,0.9]$ for all $v\in N_G^S(u_i)$ since $|\alpha_i|<0.1$. So none of the $N_G^S(u_i)$-coordinates of $\laur{p}^\prime$ will get truncated and recall that $\alpha_i$ is independent of $F_i$. Given a choice of $\bm{\alpha}$, consider now the following expression as a map of $\alpha_i$:  
\begin{align}
\label{eqn: difference of expected degrees analysis}    
f_c(\alpha_i):=\bE[d^S_{G(\laur{p})}(v_i)]-\bE[d^S_{G(\laur{p})}(v_j)]-c = (\textbf{v}_i-\textbf{v}_j)\cdot \laur{p} -c .
\end{align}
\indent Having conditioned on $\bm{\alpha}$ above, note that the event $E_{v_i,v_j}^{\ i,j}(c)$ holds only if $f_c(\alpha_i)$ lies in an interval of length $2$. To bound the probability of this happening we have to understand how $f_c$ changes as $\alpha _i$ increases. Observe the contribution from each coordinate $v$ of $\laur{p}$ to the inner product on the right hand side of \eqref{eqn: difference of expected degrees analysis} is $0$ if $v\notin N^S_G(u_i)$ and exactly $\alpha _i$ otherwise, since none of these coordinates were truncated from $\laur{p}^\prime$ when conditioning on $\overline{F_{i}}$. It follows that for $\varepsilon>0$: 
$$f_c(\alpha_i+\varepsilon)-f_c(\alpha_i)=\varepsilon\displaystyle\sum_{v\in S}\big (({\bf v}_{i})_v-({\bf v}_{j})_v \big )\cdot \mathbbm{1}_{v\sim u_i}=\varepsilon \big(|N_G^S(u_i) \cap N_G^S(v_i)|-|N_G^S(u_i) \cap N_G^S(v_j)|\big).$$ 
However, it is not hard to see that
\begin{align*}
    2|N_G^S(u_i)\cap N_G^S(v_i)| & = d_G^S(u_i) + d_G^S(v_i)-\text{div}_G^S(u_i,v_i),\\
    2|N_G^S(u_i)\cap N_G^S(v_j)| & = d_G^S(u_i) + d_G^S(v_j)-\text{div}_G^S(u_i,v_j).
\end{align*}
It follows that 
\begin{align*}
    2|N_G^S(u_i) \cap N_G^S(v_i)|-2|N_G^S(u_i) \cap N_G^S(v_j)| 
       & = 
    d_G^S(v_i) - d_G^S(v_j) - \diver ^S_G(u_i, v_i) + \diver ^S_G(u_i, v_j)\\
       & \geq 
    \diver ^S_G(u_i, v_j) - \diver ^S_G(u_i, v_i)\\
        & \geq 
    \diver ^S_G(u_i, u_j) - \diver ^S_G(u_j, v_j) - \diver ^S_G(u_i, v_i)   
        \geq 
    D,
\end{align*}
by using the hypothesis, that $d^S_G(v_i) \geq d^S_G(v_j)$ and triangle's inequality. From here we can then deduce that $f(\alpha_i+\varepsilon)-f(\alpha_i)\geq \varepsilon D/2$.\\ 
\indent Therefore, conditioned on $\bm {\alpha }$ as above, if $E^{\ i,j}_{v_i,v_j}(c)$ occurs then $\alpha_i$ lies in an interval of length at most $4/D$. This implies that ${\mathbb P}(E^{i,j}_{v_i,v_j}(c)|{\overline F}_i) \leq 2\beta^{-1}D^{-1}$. Combined with \eqref{eqn:prob-E-control} and \eqref{eqn:F-control}, this proves the lemma.
\end{proof}

Let us note that by taking $E_i=0$ and $W_i=\{u_i\}$ for each $i\in [|U|]$ in the previous lemma we recover the following result, which is similar to Lemma 4.3 from \cite{LongPL}.

\begin{cor}
    \label{cor: blended-distribution-control}
        Let $G$ be a graph, $D\geq 1,\ \beta\in(0,0.1),\ \gamma\in(0,1]$ and $U, S \subset V(G)$ such that $U$ is $\gamma $-balanced to $S$ and $|\text{div}_G^S(u,v)|\geq D$ for all $u\neq v$ in $U$. Suppose that ${\cal D}$ is the product distribution ${\cal B}_\beta(U,S) \times {\cal D}'$ 
        on $[0.1,0.9]^{V(G)}$ where ${\cal B}_\beta(U,S)$ is the blended probability distribution on $[0.1,0.9]^{S}$ and ${\cal D}'$ is any distribution on $[0.1,0.9]^{V(G)\setminus S}$. Then for all $u,v\in U$ one has:
                \begin{equation}\label{eqn:bad-eqn-control}
                    \bad {u,v} \leq \dfrac{2}{\beta D}+ \max\{d_G^S(u),d_G^S(v)\}\cdot 2\exp\left(\dfrac{-0.045}{\gamma\beta^2|U|}\right).
                \end{equation}
\end{cor}

As a quick application of the previous corollary, we obtain a proof of Theorem \ref{thm:pressure-bound}.

\begin{proof}[Proof of Theorem \ref{thm:pressure-bound}]
    To begin, we note some convenient assumptions between the parameters. Observe first that we may assume that $D \leq |U|^{3/2}$. Indeed the conclusion for larger values just asserts that $f(G) \geq \widetilde \Omega (|U|)$, which follows from the case $D = |U|^{3/2}$. We can also assume that\linebreak $|S| \leq D|U|^2$ since for larger values of $|S|$ it is easily seen that if properties (i) and (ii) from the Theorem hold, then they also hold for a subset $S' \subset S$ with this bound. Thus, working with $S'$ in place of $S$, we can assume this upper bound. Lastly, note that we can assume that $\gamma \geq 1/|U|$, as otherwise conditions (i) and (ii) cannot be satisfied.
    
    Set $\beta ^{-1} := 10\sqrt{\gamma|U|\log|U|}$ and apply Corollary \ref{cor: blended-distribution-control}. This gives
        \begin{eqnarray*}
            \bad {u,v} \leq \frac {20 \sqrt {\gamma |U| \log |U|}}{D} + |S| \cdot 2 \exp (-4.5 \log |U|) \leq \frac {40 \sqrt {\gamma |U| \log |U|}}{D},
        \end{eqnarray*}
    where the final inequality uses that $\gamma |U| \geq 1$ and that $D|S| \leq |U|^{3.5}$. Summing over all pairs of vertices in $U$, this gives $\bad {U} \leq \alpha |U|$ where $\alpha := 20D^{-1}\sqrt{\gamma|U|\log|U|}$. An application of Theorem \ref{thm: bad-control-implies-distinct-expected-degrees} then completes the proof.
\end{proof}

The last lemma of this section provides a hypothesis to maintain `bad' control when combining multiple disjoint sets with `bad' control under different well-behaved distributions.

\begin{lem}
    \label{lem:merging-control}
        Let $1 \leq m_0 \leq M \leq 2M \leq M_0$ and let $f:[m_0,M_0] \to [0, \infty )$ be a differentiable map such that $f$ is increasing, while $f'$ is decreasing on $[m_0, M_0]$. Suppose $G$ is a graph which contains disjoint vertex sets $V_1,V_2,\ldots, V_t, S$ satisfying the following properties:
            \begin{itemize}
                \item[(i)] There are sets $U_i\subset V_i$ such that $\sum _{i\in [t]} |U_i| \geq M$ with $\max _{i\in [t]}|U_i| \geq m_0$; 
                \item[(ii)] For each $i\in[t]$ there is a distribution ${\cal D}_i$ on $[0.1,0.9]^{V_i}$ such that $\baddd{{\cal D}}{i}{V_i}{U_i} \leq |U_i| f(|U_i|)$;
                \item[(iii)] There is a distribution ${\cal E}$ on $[0.1, 0.9]^S$ such that $\baddd{\cal E}{}{S}{u_i,u_j} \leq f'(M_0)$ for all $u_i \in U_i$ and $u_j \in U_j$ with $i \neq j$.
            \end{itemize}
        Then there is a set $U \subset V(G)$ with $|U|\geq M$ and a distribution ${\cal D}$ on $[0.1,0.9]^{V(G)}$ such that  $$\bad {U} \leq |U|f(|U|).$$
\end{lem}

\begin{proof}
       We first note that if $a \in [m_0, M]$ and $0 \leq b \leq a$ then $m_0 \leq a+b \leq 2M \leq M_0$ and
            \begin{align}\label{simple-qineq}
                a\cdot f(a) + b\cdot f(b) + ab\cdot f'(M_0)
                    \leq (a+b)\cdot f(a+b).
            \end{align}
        Indeed, by the mean value theorem there is $\xi \in [a, a+b] \subset [m_0, M_0]$ with the property that $f(a+b) = f(a) + b f'(\xi )$. As $f'$ is decreasing on $[m_0,M_0]$, we have $f'(M_0 ) \leq f'(\xi )$ and so
            \begin{align*}
                a f(a) + b f(b) + ab\cdot f'(M_0) 
                   & = 
                a \big( f(a) +  b\cdot f'(M_0) \big) + bf(b)\\
                    & \leq 
                a \big( f(a) + b\cdot f'(\xi )\big) + bf(b)\\
                    & \leq 
                a f(a+b) + bf(b)
                    \leq 
                (a+b)\cdot f(a+b),
            \end{align*}
        where the final inequality uses that $f$ is increasing and $b \leq a+b$.
       
       We are now able to prove the lemma. Order the sets so that $|U_1| \geq |U_2| \geq \ldots \geq |U_t|$. Note that if $|U_1| > M$ then the conclusion is trivial taking $U = U_1$ and ${\cal D} = {\cal D}_1 \times {\cal T}$, where ${\cal T}$ is the trivial distribution on $[0.1,0.9]^{V(G) \setminus U_1}$. Thus we may assume that $|U_1| \leq M$. Setting $U_{<j} := \bigcup _{i<j}U_i$ for all $j\in [2,t+1]$, by discarding sets $\{U_i\}_{i > t'}$ for some $t' \leq t$ we can  additionally assume that $|U_{<j}| < M$ for all $j\leq t$ and that $U := U_{<t+1}$ satisfies $M \leq |U| \leq 2M < M_0$. Observing that $m_0 \leq |U_1| \leq |U_{<j}| \leq M$ for $j \in [2,t]$, we can apply \eqref{simple-qineq} to obtain
            \begin{align}\label{eqn:combining-control}
                |U_{<j}|\cdot f(|U_{<j}|) + |U_j|\cdot f(|U_j|) + |U_{<j}||U_j|\cdot f'(M_0) \leq |U_{<j+1}|\cdot f(|U_{<j+1}|).
            \end{align}
        
        Let $X := V(G) \setminus \big ( S \cup \bigcup _{i=1}^t V_i \big )$ and take ${\cal D}$ to be the product distribution ${\cal D} := \prod _{i = 1}^t {\cal D}_i \times {\cal E} \times {\cal T}$ on $\prod _{i=1}^t [0.1, 0.9]^{V_i} \times [0.1, 0.9]^{S} \times [0.1, 0.9]^{X} = [0.1, 0.9]^{V(G)}$, where ${\cal T}$ is the trivial $X$-induced distribution. By Lemma \ref{lem: product-dist-lem} and (i)-(iii) from the hypothesis we obtain
    \begin{align*}
        \bad {U} 
                &= 
        \sum _{i \in [t]} \bad {U_i} + \sum _{\{i,j\} \subset [t]} \bad {U_i, U_j}\\
                & \leq 
        \sum _{i \in [t]} \baddd {\cal D}{i}{V_i} {U_i} 
                + 
        \sum _{\{i,j\} \subset [t]} \baddd {\cal E}{}{S}{U_i, U_j}\\
                & \leq 
        \sum _{i \in [t]} \baddd {\cal D}{i}{V_i} {U_i} 
                + 
        \sum _{\{i,j\} \subset [t]} |U_i|\cdot |U_j| \cdot \max _{(v_i,v_j) \in U_i \times U_{j}} \left\{
        \baddd {\cal B}{}{S}{v_i, v_j}\right\} \\
                & \leq 
        \sum _{i \in [t]} |U_i|\cdot f(|U_i|)
                + 
        \sum _{\{i,j\} \subset [t]} |U_i|\cdot |U_{j}| \cdot f'(M_0)  \\ 
            &= 
        \sum _{i\in [t]} |U_{i}|\cdot f(|U_{i}|)  
            \ + f'(M_0)  \cdot \sum _{j\in [t]}  |U_{<j}| \cdot |U_{j}|  \leq 
        {|U|\cdot f(|U|)}, 
    \end{align*}
where the final inequality follows by repeatedly applying \eqref{eqn:combining-control}. This completes the proof.
\end{proof}

\section{Cluster neighbourhoods}\label{sect:cluster-neighbourhoods}

In order to prove Theorem \ref{thm: DDandHS_thm} it will be crucial to understand how vertex neighbourhoods correlate, or cluster, in the graph. The following definition gives a useful measurement of this neighborhood clustering.

\begin{dfn}
    Let $G$ be a graph, let $v$ be a vertex of $G$ and let $S\subset V(G)$. Given $M > 1$ and $t \in {\mathbb N} \cup \{0\}$, the $(M, t)$-\textbf{cluster neighbourhood} of $v$ to $S$ is defined as 
        \begin{align*}
            W_t^S(v; M)
                := 
            \bigg \{u\in V(G): \big|N_G^S(u)\Delta N_G^S(v)\big|\leq \bigg( \frac {4 ^t}{M} \bigg ) \cdot \big|N_G^S(v)\big| \bigg \}.
        \end{align*}
    When $S$ and $M$ are clear from the context we simply write $W_t(v)$ or $W^S_t(v)$.
    
    Given $\lambda > 1$, setting $\Theta := (M, \lambda)$, 
    the $\Theta $-\textbf{moment of} $v$ \textbf{to} $S$ is given by 
        \begin{align*}
            T_{\Theta}^S(v) 
                := 
            \min \bigg \{ t \in {\mathbb N} \cup \{0\}:
            |W^{S}_{t}(v; M)| \leq \lambda \cdot |W^{S}_{t}(v; M)| \bigg \}.
        \end{align*}
    The $\Theta $-\textbf{cluster} of $v$ to $S$ is then given by $W_{T_{\Theta }(v)}^S(v; M)$. Once again, when $S, M$ are clear from the context we write $W_*(v) := W_{T_{\Theta }(v)}^S(v; M)$ and $W_+(v) := W_{T_{\Theta }(v)+ 1}^S(v; M).$
\end{dfn}

In the next subsection we present some useful properties of cluster neighbourhoods. The second subsection then proves a key lemma, which allows a partial decomposition of a graph using its cluster neighbourhoods. This result will be central in our proof of Theorem \ref{thm: DDandHS_thm}.

\subsection{Simple properties of cluster neighbourhoods}

The next lemma captures several useful facts about cluster neighbourhoods and $\Theta $-clusters. 

\begin{lem}\label{lem:simple-cluster-properties}
    Given $G, S$ and $\Theta = (M, \lambda )$ as above, the following hold: 
        \begin{itemize}
            \item [(i)] $W^S_t(v)\subset W^S_{t+1}(v)$ for all $t\geq 0$.
            \item [(ii)] $|W_*(v)| = |W^S_{T_{\Theta}(v)}(v)| \geq \lambda ^{T_{\Theta }^S(v)}$.
            \item [(iii)] $T_{\Theta }^S(v) \leq {\log _{\lambda }|S|} \leq 
            {\log _{\lambda }|G|}$.
            \item [(iv)] $|W_+(v)| \leq \lambda |W_{*}(v)|$.
        \end{itemize}            
\end{lem}

\begin{proof}
    Part (i) is clear from the definition of $W^S_t(v)$. Parts (ii) and (iv) hold by the definition of $T_{\Theta }^S(v)$, noting that $v \in W_0^S(v)$. Lastly, part (iii) follows from (ii) since $|W_*(v)| \leq |S| \leq |G|$.
\end{proof}

\indent We next present the following simple, yet useful lemmas.

\begin{lem}\label{lem:cluster-bound}
    Let $G, S$ and $\Theta = (M , \lambda )$ be as above. Then for any vertex $v$ and non-negative integer $t< \log_{4} M-1$ one has $|W^S_t(v)| \leq 2\Delta (G)$.
\end{lem}

\begin{proof}
    Fix $v\in V(G)$ and double count the edges (say there are $e_v$ of them) with one endpoint in $W_t^S(v)$ and the other in $N^S(v)$. On one hand, each vertex in $N^S(v)$ contributes at most $\Delta(G)$ edges, hence $e_v\leq \Delta(G)\cdot d^S(v)$. On the other hand, each vertex in $W_t^S(v)$ must be adjacent to at least $d^S(v)(1-4^tM^{-1})$ vertices in $N^S(v)$, thus $e_v\geq |W^S_t(v)|\cdot d^S(v)(1-4^tM^{-1})$. The result follows by combining these two bounds and using that $4^t < M/2$. 
\end{proof}

\begin{lem}\label{lem:disjoint-clusters} 
    Let $G, S$ and $\Theta = (M, \lambda )$ be as above. Let $t_1,t_2\in \mathbb{N}$ and suppose $v_1,v_2\in V(G)$ such that $v_2\notin W_{t_1+1}^S(v_1)$ and $3\cdot 4^{t_1} d^S(v_1) \geq 4^{t_2} d^S(v_2)$. Then 
    $W_{t_1}^S(v_1) \cap W_{t_2}^S(v_2)=\emptyset$. 
\end{lem}

\begin{proof}
    Assume by contradiction there is $u\in W_{t_2}^S(v_2)\cap W_{t_1}^S(v_1)$. Then, by triangle inequality
        \begin{align*}
            |\text{div}_G^S(v_1,v_2)|\leq |\text{div}_G^S(v_1,u)|+|\text{div}_G^S(v_2,u)|\leq M^{-1}\big(4^{t_1}d^S(v_1)+4^{t_2}d^S(v_2)\big)\leq 4M^{-1}4^{t_1}d^S(v_1), 
        \end{align*}
    which means that $v_2\in W_{t_1+1}^S(v_1)$, contradicting our hypothesis. Therefore our assumption is false and so $W_{t_1}^S(v_1) \cap W_{t_2}^S(v_2)=\emptyset$. 
\end{proof}

\begin{cor}\label{cor:disjoint-clusters}
    Let $G,S$ and $\Theta = (M, \lambda )$ be as above. Let $t_1,t_2\in \mathbb{N}$ and suppose $v_1,v_2\in V(G)$ such that $v_{3-i}\notin W_{t_i+1}^S(v_i)$ for both $i\in \{1,2\}$. Then $W_{t_2}^S(v_2)\cap W_{t_1}^S(v_1)=\emptyset$.
\end{cor}

\begin{proof}
    Relabelling if necessary, we can assume that $4^{t_1}d^S(v_1) \geq 4^{t_2}d^S(v_2)$. The degree condition in Lemma \ref{lem:disjoint-clusters} then holds, implying the result.
\end{proof}

We note the following short lemma shows that having many clusters of small size and high enough degree in a graph guarantees a certain level of diversity. It also serves as a preamble to the next result, which states that otherwise vertices cluster around some centers of mass.

\begin{lem}\label{lem:small-weight}
    Let $G, S$ and $\Theta = (M, \lambda )$ be as above. Let $n:=|V(G)|$ and suppose there is a set $A\subset V(G)$ with $|A| = m \geq \delta n \geq 1$ such that $|W_t^S(v)|\leq n/m$ and $d_G^S(v)\geq d$ for all $v\in A$. Then there is $U\subset A$ with $|U|\geq \delta m/2$ such that $|\diver ^S(u,v)|\geq 4^td/M$ for all $u\neq v$ in $U$. 
\end{lem}

\begin{proof}
    Let us build a new graph $H$ on the vertex set $A$ by drawing an edge between two vertices $u$ and $v$ if $|\diver ^S(u,v)|< 4^td/M$. Now, given $v\in A$, let us observe that $N_H(v)\subset W_t^S(v)$, therefore the average degree of $J$ is at most $n/m$. By Tur\'an's theorem we find an independent set $U$ in $H$ with $|U|\geq |A|(n/m+1)^{-1}\geq m|A|/(n+m)\geq \delta m/2$, which is the desired set.  
\end{proof}

\subsection{Partial decomposition through cluster neighbourhoods}

This subsection is dedicated to a central result in our proof of Theorem \ref{thm: DDandHS_thm}. Although the statement is technical, roughly speaking, it gives conditions under which we are able to use cluster neighbourhoods in a graph $G$ to obtain disjoint sets $V_1,\ldots, V_t, S$, such that:
\begin{itemize}[nosep] 
    \item $\bigcup _{i\in [t]} V_i$ covers a substantial portion of $V(G)$,
    \item the neighbourhoods of vertices within each $V_i$ to $S$ are similar,
    \item the neighbourhoods of vertices within distinct $V_i$ to $S$ differ substantially,
    \item the overall contribution of degrees from $V_i$ to $S$ is balanced (in a certain technical sense).
\end{itemize}
We proceed to the statement.\vspace{2mm}

\begin{lem}\label{lem:cluster-partition-lemma}
    Let $G$ be an $n$-vertex graph with $n\geq 2^{32}$. Suppose that $k \in {\mathbb N}$, $\lambda, M \geq 2$, $\alpha > 0$ such that $\log_2 ^2n \leq k \leq n\log_2 ^{-3}n$ and $0 \leq \alpha \leq (\lambda \log^5n )^{-1}$, and that $\Theta = (M, \lambda )$. Suppose that
    \begin{align*}
        A 
            \subset  
        \Big \{v \in V(G): M \log ^2 n \leq d_G(v) \leq n/2 \mbox{ and } \big |W_{T_{\Theta }(v)}^{V(G)}(v;M) \big | \leq \alpha n \Big \}
            \subset 
        V(G)
    \end{align*}
    satisfies $|A|\geq n/8$. Then there is a set $U = \{u_1,\ldots, u_t\} \subset A$, a collection of pairwise-disjoint sets $V_1,V_2,\ldots, V_t, S \subset V(G)$ and a set $\{d_i\}_{i\in [t]} \subset [0, n]$ satisfying the following properties:
        \begin{itemize}
            \item [(i)] $u_i \in V_i$ and $|V_i| \leq \alpha n$ for all $i\in [t]$;
            \item [(ii)] $t \leq k$. Furthermore, if $t < k$ then $\sum _{i\in [t]} |V_i| \geq \frac{n}{1000 \lambda \log ^2n}$;            
            \item [(iii)] $\diver ^S_G(u_i, v_i) \leq d_i$ for all $v_i \in V_i$ and $i\in [t]$;
            \item [(iv)] $\diver ^S_G(u_i, u_j) \geq (5M)^{-1}\cdot \max\big (d_G(u_i), d_G(u_j) \big ) + d_i + d_j$ for all $i\neq j$ in $[t]$;
            \item [(v)] The set $U$ is $\gamma $-balanced to $S$, where $\gamma := \log ^5n \cdot \max \big (n^{-1}\cdot\Delta (G),\ t^{-1} \big )$.
        \end{itemize}
\end{lem}

\begin{proof}
To begin with, write $W_*(v) = W_{T_{\Theta}(v)}^{V(G)}(v; M)$ and $W_+(v) = W_{T_{\Theta}(v) + 1}^{V(H)}(v; M)$ for all $v \in A$. By the pigeonhole principle there are $L \in [n]$, $T \in [\log _{\lambda }(n)] \subset [\log_2 n]$ and a set $B \subset A$ with $|B| \geq |A|/\log ^2n $ such that $|W_*(v)| \in [L, 2L]$ and $T_{\Theta }(v) = T$ for all $v \in B$. Also $L \leq \alpha n$ by the definition of $A$, whereas Lemma \ref{lem:simple-cluster-properties} (iv) gives $|W_+(v)| \leq \lambda |W_*(v)| \leq 2 \lambda L$.\vspace{2mm}

Next, we take a random partition $V(G) = R \cup S$ by adding each $v\in V(G)$ independently to $R$ with probability $3/4$ and to $S$ with probability $1/4$. Set $p := \min \big(32k|B|^{-1}, (4\lambda L)^{-1}\big)\leq 0.2$ and select a random subset $U'$ of $R$ by including each vertex independently with probability $p' := 4p/3<0.3$. Stepping back, it is easy to see that elements of $U'$ are selected from $V(G)$ independently with probability $p$ and, moreover, $U'\cap S=\emptyset$. 

We now consider the following set 
    \begin{align*}
        U := \Big \{u \in B \cap U': W_+(u) \cap U' = \{u\} \mbox{ and } |W_*(u) \cap R| \geq |W_*(u)|/2 \Big \},
    \end{align*}
and the following three related events:
    \begin{itemize}
        \item ${\cal A}_1$ is the event that $|U| \in \big[2^{-5}p|B|, 2p|B|\big]$;
        \item ${\cal A}_2$ is the event that $|N_H(v)\cap U| \leq  \log ^2 n \cdot \max \big(1, p\cdot \Delta (G)\big)$ for all $v \in V(G)$;
        \item ${\cal A}_3$ is the event that $\diver _G^S(u,v) = \left(\dfrac{1}{4} \pm 0.05\right)\diver _G(u,v)$ for all vertices $u, v \in V(G)$ with $\diver _G(u, v) \geq 2^{15} \log n$. 
    \end{itemize}
    
    \noindent We will now show that these events can occur simultaneously.\vspace{3mm}
    
\noindent \textbf{Claim:} ${\mathbb P}({\cal A}_1 \cap {\cal A}_2 \cap {\cal A}_3) > 0$. \vspace{3mm}

To see this, let us first look at $\bP(\mathcal{A}_1)$ and note that $|U'| \sim \mbox{Bin}(|B|, p)$, thus:
    \begin{align*}
        {\mathbb E}\big[|U'|\big] = p|B| \geq \min \big(32k , |B|\cdot (4\lambda L)^{-1}\big) &\geq \min \big(32k, n(32\lambda L\log ^2n)^{-1}\big) \geq \log ^2 n,
    \end{align*}
where we have used that $k \geq \log ^2n$ and that $L \leq \alpha n \leq n (\lambda \log ^5n)^{-1}$.
Since $U \subset U'$, by applying Theorem \ref{cher} (Chernoff's inequality) we get:
    \begin{align}
    \label{eqn:prob-|U|-lower-bound}
        {\mathbb P}(|U| > 2|B|p) 
            \leq 
        {\mathbb P}(|U'| > 2|B|p) 
            \leq 
        \exp (- |B|p/4) 
            \leq 
        \exp (-\log ^2n /4) < n^{-3} < p/64.
    \end{align}
\indent Next, for each $v\in B$ let ${\cal E}_v$ denote the event that $\big|R \cap (W_*(v)\setminus \{v\})\big| \geq \big|W_*(v)\setminus \{v\}\big|/2$. As $\big|R \cap (W_*(v)\setminus \{v\})\big| \sim \mbox{Bin}\big(|W_*(v)|-1, 3/4\big)$ we have ${\mathbb P}({\cal E}_v) \geq 1/4$ by Theorem \ref{thm:quarter-binomial}. This gives
    \begin{align*}
        {\mathbb P}\big(v \in U\big) 
            & \geq 
        {\mathbb P}\big(v \in U'\big)\cdot {\mathbb P} \big( U' \cap (W_+(v) \setminus \{v\})  = \emptyset\ \big| \ {\cal E}_v \big) \cdot {\mathbb P}({\cal E}_v)\\
            & \geq 
p\cdot (1-p)^{|W_*(v)|}\cdot \frac{1}{4} \geq \frac{p (1-p)^{2\lambda L}}{4} \geq \frac{pe^{-4\lambda Lp}}{4} \geq \frac{pe^{-1}}{4} \geq \frac{p}{ 16},
    \end{align*}
since $1- x \geq e^{-2x}$ for $x \in [0, 1/2]$ and $p\lambda L\leq 1/4$.
Therefore ${\mathbb E}\big[|U|\big] \geq p|B|/16$. As $|U| \leq |B|$, by Markov we get ${\mathbb P}\big(|U| \geq p|B|/32\big) \geq p/32$. Combined with \eqref{eqn:prob-|U|-lower-bound}, this gives
    \begin{align*} \label{eqn:|B|-control}
        {\mathbb P}({\cal A}_1) = {\mathbb P}\Big ( |U| \in \big [2^{-5}p|B|, 2p|B| \big ] \Big ) \geq p/64.
    \end{align*}
\indent To estimate $\bP(\mathcal{A}_2)$, given $v \in V(G)$ we note that $|N_G(v) \cap U'| \sim \mbox{Bin}(|N_G(v)|, p)$ and that ${\mathbb E}\big [ |N_G(v) \cap U'| ] = p\cdot |N_G(v)| \leq p\cdot\Delta (G)$. By Theorem \ref{binomial-bound} we obtain that
    \begin{align*}
        {\mathbb P} \big (|N_G(v) \cap U'| \geq K \big ) 
        \leq (e\Delta (G)p/K)^{K} \leq (e/\log ^2 n)^{\log ^2n} \leq n^{-4},
    \end{align*}
where $K:= \log ^2 n \cdot \max \big(1,p\cdot \Delta (G)\big)$. In particular, as $|N_G(v) \cap U| \leq |N_G(v) \cap U'|$, we get 
    $${\mathbb P}({\cal A}_2) \geq {\mathbb P} \big (|N_G(v) \cap U'| \leq K \mbox{ for all } v \in V(G) \big ) \geq 1 - n \cdot n^{-4}  \geq 1 - p/256.$$ 
\indent Lastly, to lower bound ${\mathbb P}({\cal A}_3)$, note that given $u, v \in V(G)$ with $\diver _G(u,v) \geq 2^{15} \log n$, we have $\diver _G^S(u, v) \sim \mbox{Bin}\big(\diver _G(u,v), 1/4\big)$ and ${\mathbb E}\big[\diver _G^S(u,v)\big] = \diver _G(u,v)/4 \geq 2^{13}\log n$. Thus
    \begin{equation*}
        {\mathbb P}\big ( |\mbox{div}_G^S(u,v) - \mbox{div}_G(u,v)/4| 
        \geq 0.05\cdot \mbox{div}_G(u,v) \big ) \leq \exp \big ( - 0.05 ^2 (2^{13}\log n )/4 \big ) \leq n^{-5},
    \end{equation*}
by using Chernoff's inequality. We then deduce that ${\mathbb P}({\cal A}_3) \geq 1 - n^2 (n^{-5}) \geq 1 - p/256$.

Combining our estimates, we get ${\mathbb P}({\cal A}_1 \cap {\cal A}_2 \cap {\cal A}_3)\geq \bP(\mathcal{A}_1)-
\bP\left(\ov{\mathcal{A}_2}\right)-\bP\left(\ov{\mathcal{A}_3}\right)> p/128 > 0$, as claimed. We thus fix a choice of $R,S,U$ and $U'$ such that ${\cal A}_1 \cap {\cal A}_3 \cap {\cal A}_3$ holds. \vspace{4mm}

We are now in a position to select $U, V_1,\ldots, V_t$ and $S$ as in the statement of the Lemma. We have already chosen the set $S$. As ${\cal A}_1$ holds, we may also assume, by discarding some elements, that $U = \{u_1,\ldots, u_t\}$ where $t = \lceil 2^{-5}p|B|\rceil $. For all $i\in [t]$ we let $V_i := W_*(u_i) \cap R$. Note that, by definition of $U$ we have $u_i \notin W_+(u_j)$ for all $i\neq j$ in $[t]$, which by Corollary \ref{cor:disjoint-clusters} implies that $V_i \cap V_j \subset W_*(u_i) \cap W_*(u_j) = \emptyset $. Therefore the sets $V_1,\ldots, V_t$ are pairwise disjoint. As $V_i \subset R$ for all $i\in [t]$, we also have $V_i \cap S = \emptyset $. 

We can finally prove that (\emph{i})-(\emph{v}) hold for these sets. As $u_i\in A$ we have $|V_i| \leq |W_*(u_i)| \leq \alpha n$, giving (\emph{i}). Recalling the definition of $p$ we see that $t = \lceil p|B|/32\rceil  = \min \big(k, \lceil |B|(32\lambda L)^{-1} \rceil \big) \leq k$. Furthermore, if $t < k$ then $t \geq |B|\cdot (32\lambda L)^{-1}$. Besides,  $|V_i| = |W_*(u_i)\cap R| \geq |W_*(u_i)|/2 \geq L/2$ since $u_i\in U$. We therefore deduce the following inequality, which shows that (\emph{ii}) holds: 
    \begin{align*}
        \sum _{i\in [t]} |V_i| \geq \frac{tL}{2} \geq \frac{|B|}{32 \lambda L}\cdot \frac{L}{2} \geq \frac{|A|}{64 \lambda \log ^2n } \geq \frac {n}{1000 \lambda \log ^2 n}.
    \end{align*}
\indent To confirm (\emph{iii}) and (\emph{iv}), fix $i\neq j$ with $d_G(u_i) \geq d_G(u_j)$. As $u_i, u_j \in U$, from the definition of $U$ we deduce that $u_j \notin W_+(u_i)$. Since $u_i\in A$, it follows that
    \begin{align*}
        \diver _G(u_i, u_j) \geq  \frac{4^{T+1}}{M} \cdot d_G(u_i) \geq \frac{4}{M}\cdot M\log^2 n \geq 2^{15} \log n.
    \end{align*} 
Moreover, given $v_i\in V_i$ and $v_j \in V_j$, we have $v_i \in W_*(u_i)$ and $v_j \in W_*(u_j)$. Thus
    \begin{align*}
        \diver _G(u_i, v_i)  \leq  \frac {4^{T}}{M}\cdot d_G(u_i)\ \ \ \ \textbf{ and }\ \ \ \
        \diver _G(u_j, v_j)  \leq  \frac{4^T}{M} \cdot d_G(u_j) \leq \frac {4^{T}}{M} \cdot d_G(u_i).
    \end{align*}
Moving to $S$, since $d_G(u_i) \geq M \log ^2n$ and ${\cal A}_2$ holds, given $i\in [t]$ and $v_i \in V_i$ we have
    \begin{align*}
        \diver ^S_G(u_i, v_i) 
            \leq 
        \max \big ( (0.25 + 0.05)\cdot \diver _G(u_i, v_i) ,\ 2^{15} \log n \big ) 
            \leq 
        0.3 \cdot \frac {4^T}{M} \cdot d_G(u_i),
    \end{align*}
Continuing along these lines, as $\diver _G(u_i, u_j) \geq 2^{15}\log n$ for $i \neq j$, we obtain
    \begin{align*}
        \diver _G^S(u_i, u_j)
            & \geq 
        (0.25 - 0.05)\cdot \diver _G(u_i, u_j)
            \geq 
        0.8 \cdot \frac {4^{T}}{M} \cdot d_G(u_i) \\
            & \geq 
        \frac{0.2}{M}\cdot d_G(u_i) + 
        \diver _G^S(u_i, v_i) + 
        \diver _G^S(u_j, v_j).
    \end{align*}
Thus (\emph{iii}) and (\emph{iv}) hold by taking $d_i = 0.3\cdot 4^T\cdot d_G(u_i)/M$ and $d_j = 0.3\cdot 4^T\cdot d_G(u_j)/M$.

It only remains to check that (\emph{v}) is satisfied when $\gamma = \log ^5n \cdot \max (\Delta /n, 1 / t )$. The statement follows since ${\cal A}_1$ and ${\cal A}_3$ hold, because for all $v \in S\subset V(G)$ we have 
    \begin{align*}
        |N_H(v) \cap U| 
            & \leq 
        \log ^2n \cdot  \max\big(1, p\cdot \Delta (G)\big) 
            = 
        \log ^2n \cdot \max \big(|U|^{-1}, p\cdot \Delta (G)\cdot |U|^{-1}\big) \cdot |U|\\ 
            & \leq 
        \log ^5 n \cdot \max \big( t^{-1}, n^{-1}\cdot \Delta (G)\big) \cdot |U| 
            = 
        \gamma |U|,
    \end{align*}
where the final inequality uses $|U| \geq p|B|/32 \geq 2^{-5}|A|p\log ^{-2}n \geq 2^{-8}pn\log ^{-2}n \geq pn\log ^{-3}n$. With this, we have completed the proof of the lemma.
\end{proof}

\section{Main Theorem}\label{sec:proof}

The focus of this final section will be on the proof of Theorem \ref{thm:main-thing} below. Our main result, Theorem \ref{thm: DDandHS_thm}, follows immediately from combining this result with Theorem \ref{thm: bad-control-implies-distinct-expected-degrees}.

\begin{thm}\label{thm:main-thing}
   There is $C>1$ such that the following holds with $g_1(x) := \exp ( C \cdot (\log _2 x)^{2/3} )$. Suppose that $G$ is an $n$-vertex graph and let $k \in [1,\infty)$  such that $n\leq k^2$ and $\hom (G) \leq n^2/k^3$.\linebreak Then there is a set $U\subset V(G)$ with $|U|\geq k/g_1(k)$ and a probability distribution $\mathcal{D}$ on $[0.1,0.9]^{V(G)}$ with $\bad{U}\leq |U|\cdot g_1(|U|)$.
\end{thm}

As mentioned in the Introduction, our proof will be inductive. Here we will also sometimes require the following alternative result, which appeared\footnote{The presentation here differs very slightly from Theorem 5.1 in \cite{LongPL} in that the hypothesis there was that $k\geq 1$, $n\geq 20000k^2$ and $\hom(G) \leq n/25k$. However, the statement in Theorem \ref{thm:LP-thm} follows from this by applying this result with $k':= k/200$ in place of $k$ (taking $C$ large enough to cover the small regime when $k'<1$).} as Theorem 5.1 in \cite{LongPL}.

\begin{thm}\label{thm:LP-thm}
   There is $C>1$ such that the following holds with $g_2(x) := C (\log _2 x)^2$. Suppose that $G$ is an $n$-vertex graph and let $k \in [1,\infty)$ such that $n\geq k^2$ and $\hom (G) \leq n/k$. Then there is a set $U\subset V(G)$ with $|U|\geq k/g_2(k)$ and a probability distribution $\mathcal{D}$ on $[0.1,0.9]^{V(G)}$ with $\bad{U}\leq |U|\cdot g_2(|U|)$.
\end{thm}


\subsection{Proof of Theorem \ref{thm:main-thing}} \label{subsect:proof-of-main-thing}

We will prove the theorem by induction on $n$. Note that we can assume $n \geq k^{3/2}$, as otherwise the hypothesis gives $\hom (G) \leq n^2/k^3 < 1$ which is never satisfied if $n\geq 1$.

For the induction it is convenient to prove the theorem taking $g_1(x) := C_1 \exp ( C_2 (\log x)^{2/3})$ for constants $C_1, C_2 >1$ instead; 
 it is easily seen the statement given in Theorem \ref{thm:main-thing} follows quickly from this by letting $C:=C_2\log C_1$. Observe that $g_1$ is increasing on $[1, \infty)$ and that its derivative 
    $$g_1'(x) = \frac{2C_1C_2}{3} \cdot \frac {\exp \big(C_2(\log x)^{2/3}\big)}{x (\log x)^{1/3}}$$ 
is decreasing on $[x_0, \infty )$, where $x_0 := e^{C_2^{3}}$. By taking $C_1$ sufficiently large we can assume $n$ (and hence $k$, as $k^2\ge n$) is sufficiently large 
for our estimates below, and that $\log k \geq \frac {1}{2} \log n > x_0$, which will be helpful in applying Lemma \ref{lem:merging-control} below.

As a hypothesis for our induction, we will assume the result holds for all $G[V]$ with $V \subsetneq V(G)$. To take advantage of this, we introduce the function $k: {\mathcal P}\big ( V(G) \big ) \to [0, \infty)$ given by 
    \begin{equation*}
        k(V) := \begin{cases}
                    k \cdot \big ( \frac{|V|}{n} \big )^{2/3} & \mbox{if } |V|^{1/2} \geq {n^2}/{k^3},\\
                    |V| \cdot \big (\frac{k^3}{n^2} \big ) & \mbox{if } |V|^{1/2} \leq {n^2}/{k^3}.
                \end{cases}
    \end{equation*}
As any non-trivial subset $V \subsetneq V(G)$ satisfies $\hom (G[V]) \leq \hom (G) \leq n^2/k^3$, it follows that:
    \begin{itemize}
        \item [{\hypertarget{O1}{\textbf{(O1)}}}] If $|V| \in [n^4/k^6, n)$ then $|V|^{1/2} \geq n^2/k^3$, hence $k(V) = k(|V|/n)^{2/3}$. Thus $|V| \leq \big (k(V) \big )^2$ and $\hom (G[V]) \leq n^2/k^3 = |V|^2/k(V)^3$, so by induction there is $U \subset V$ and a distribution ${\cal E}$ on $[0.1,0.9]^{V}$ such that $|U| \geq k(V)/ g_1(k(V))$ and $\baddd {{\cal E}}{}{}{U} \leq |U| \cdot g_1(|U|)$.
        \item [{\hypertarget{O2}{\textbf{(O2)}}}] If $|V| \leq n^4/k^6$ then $|V|^{1/2} \leq n^2/k^3$, hence $k(V) = |V|(k^3/n^2)$. Thus $ |V| \geq \big ( k(V) \big )^2$ and $\hom (G[V]) \leq n^2/k^3 = |V|/k(V)$, so by Theorem \ref{thm:LP-thm} there is $U \subset V$ and a distribution ${\cal E}$ on $[0.1,0.9]^{V}$ such that $|U| \geq k(V)/  g_2(k(V))$ and $\baddd {{\cal E}}{}{}{U} \leq |U|\cdot g_2(|U|)$.
    \end{itemize}

 We need to show that there is a set $U \subset V(G)$ with $|U| \geq k/g_1(k)$ and a distribution ${\cal D}$ on $[0.1, 0.9]^{V(G)}$ with $\bad {U} \leq |U|\cdot g_1(|U|)$. The following strategy will be key in the proof.\vspace{3mm}

\noindent \hypertarget{strategy}{\textbf{\underline {Strategy:}}} First find disjoint sets $V_1,V_2, \ldots, V_t, S \subset V(G)$ and a distribution ${\cal E}$ on $[0.1,0.9]^{S}$. Given $i\in [t]$, we  apply either \hyperlink{O1}{ \textbf{(O1)}} or \hyperlink{O2}{ \textbf{(O2)}} to $V_i$ (depending on $|V_i|$) to obtain sets $U_i \subset V_i$ and distributions ${\cal D}_i$ on $[0.1, 0.9]^{V_i}$ such that:
\begin{itemize}[nosep]
    \item $|U_i| \geq k(V_i)/ g_1(k(V_i))$ and $\baddd {D}{i}{V_i}{U_i} \leq |U_i|\cdot g_1(|U_i|)$ if \hyperlink{O1}{ \textbf{(O1)}} applies;\vspace{2mm}
    \item $|U_i| \geq k(V_i)/ g_2(k(V_i))$ and $\baddd {D}{i}{V_i}{U_i} \leq |U_i|\cdot g_2(|U_i|)$ if \hyperlink{O2}{ \textbf{(O2)}} applies.
\end{itemize} 
This naturally results in a partition $[t] = I_1 \sqcup I_2$ according to which \textbf{(Oa)} with $a\in \{1,2\}$ applies. We can then complete the proof with Lemma \ref{lem:merging-control} (with parameters in that lemma as $m_0 = x_0$, $M = k/g_1(k)$ and $M_0 = 2k$) provided we can ensure that:
    \begin{itemize}
        \item [\textbf{(A)}] There is a distribution ${\cal E}$ on $[0.1, 0.9]^S$ such that $\baddd {E}{}{S}{u_i, u_j} \leq g_1'(2k) $ for all $u_i \in U_i$ and 
        $u_j \in U_j$ with $i \neq j$;\vspace{2mm}
        \item [\textbf{(B)}] $\sum _{i\in I_j} |U_i| \geq k/ g_1(k)$ for some $j \in \{1,2\}$ with $\max _{i\in I_j}|U_i| \geq x_0$.\vspace{3mm}
    \end{itemize}
    
Before proceeding, we fix the following parameters, as well as $\Theta = (M, \lambda )$:
    \begin{align}\label{eqn:parameter-spec}
        \log_2\lambda = (\log_2 k)^{\frac{4}{9}},\ \  \log_2 T = (\log_2 k)^{\frac{5}{9}},\ \ \log_2 M=(\log_2 k)^{\frac{2}{3}}. 
    \end{align}

    To begin with, let $A = \{v \in V(G): d_G(v) < n/2\}$. Note that we may assume that $|A| \geq n/2$; otherwise we simply work with $\ov{G}$ instead, as the hypothesis and the conclusion of Theorem \ref{thm:main-thing} do not change under taking complement. Given $\Theta $ as above, we will write $W_*(v)$ to denote the $\Theta $-cluster of $v$ in $G$. We consider the partition $A = A_1 \sqcup A_2 \sqcup A_3$ given by:
        \begin{align*}
            A_1 & := \big \{ v \in V(G) : d_G(v) \in (k^{3/2}/T, n/2) \mbox { with } |W_*(v)| \geq 4k \big \},\\
            A_2 & := \big \{ v \in V(G) : d_G(v) \in (k^{3/2}/T, n/2) \mbox { with } |W_*(v)| < 4k \big \},\\
            A_3 & := \big \{ v \in V(G) : d_G(v) < k^{3/2}/T \big \}.
        \end{align*}
    Our proof proceeds according to the size of these sets.\vspace{3mm} 

    \noindent \underline{\textbf{Case 1}}: There is $v_0 \in A_1$.\vspace{3mm}

     Let $t = T_{\Theta }(v_0)$. We start by taking a $4k$-element subset $S_0 \subset W_*(v_0) = W_{t}(v_0)$, which is possible as $v_0 \in A_1$, and define: 
        $$Y:=\{v\in N(v_0): d_G^{S_0}(v)\leq 3k\} \quad \mbox{ and } \quad 
        \ov{Y}:=\{v\notin N(v_0):d_G^{S_0}(v)\geq k\}.$$ 
    We will show that neither of these sets is too large and then apply induction to both sets $V_1:=N(v_0)\setminus (Y\cup S_0)$ and $V_2:=V(H)\setminus (N(v_0)\cup \ov{Y}\cup S_0)$, combining them together and using concavity to show that (despite the small loss of vertices) we get a set of the desired size.
    
    To see this, we double count the non-edges between $N(v_0)$ and $S_0$. Recall that $S_0\subset W_{t}(v_0)$, so there are at most $|S_0|(4^td_0/M) = (4k)(4^td_0/M)$ such non-edges. However each $v\in Y$ sends at\linebreak least $k$ non-edges to $N(v_0)$, so this is at least $|Y|k$, giving $|Y| \leq 4^{t+1}d_0/M$. By similarly double counting edges between $V(H)\setminus N(v_0)$ and $S$ we get $|\ov{Y}|\leq 4^{t+1}d_0/M$. As a consequence, 
    $$|V_1| \geq d_0 - 4^{t+1}d_0/M - 4k \geq 15d_0/16 \quad \mbox{ and } \quad |V_2|\geq n-d_0 - 4^{t+1}d_0/M - 4k \geq n-17d_0/16,$$
    using $4^{t+1}/M \leq 2^{-5}$, which is true since $t = T_{\Theta }(v) \leq \log _{\lambda }n \leq 2\log _{\lambda} k \leq \log _4M - 6$ by Lemma \ref{lem:simple-cluster-properties} (\emph{iii}) and also $d_0/2^5 \geq k^{3/2}/2^5T \geq 4k$.
    
    First observe that $d_G^S(u_1) \geq d_G^S(u_2) + 2k$ for all $u_1 \in U_1$ and $u_2 \in U_2$ by construction, thus, by Lemma \ref{lem: uniform-distribution-control} taking ${\cal E} = {\cal U}_S$ we obtain $\baddd{{\cal E}}{}{S}{u_1, u_2} \leq 3/2k\leq \exp (\frac{C_2}{2}(\log k)^{2/3} )/k\leq g_1'(2k)$ for all $u_1 \in U_1$ and $u_2 \in U_2$. This provides \textbf{(A)} from the \hyperlink{strategy}{\textbf{Strategy}}.
    
    It remains to show that \textbf{(B)} also holds. To do this, we will split the argument into two subcases according to the value of $d_G(v_0)$.\vspace{3mm}

    \noindent \underline{\textbf{Case 1(a)}}: $d_G(v_0) = d_0\geq 2n^4/k^{6}$.\vspace{3mm} 
    
    Since $d_0 \leq n/2$ we have $n-d_0 \geq d_0$ and so $|V_i| \geq d_0 - (4^{t+1}/M)d_0 - 4k \geq d_0/2  \geq n^4/k^6$ for $i \in \{1,2\}$. Setting $k_i := k(V_i)$ for $i = 1,2$ we find via the induction in \hyperlink{O1}{ \textbf{(O1)}} a subset $U_i \subset V_i$ and a distribution $\mathcal{D}_i$ on $[0.1,0.9]^{V_i}$ such that $|U_i| \geq k_i/g_1(k_i)$ and $\baddd{D}{i}{}{U_i} \leq |U_i|\cdot g_1(|U_i|)$.
    
   \indent Next, note that $|V_2|/n\geq (n-17d_0/16)/n>0.46$. As the map $x\to x^{2/3}-(1+3x)/4$ is positive for $x\in (0.46,1)$, we deduce that
   $$\left(\dfrac{|V_2|}{n}\right)^{\frac{2}{3}}\geq \dfrac{1}{4}\left(1+\dfrac{|V_2|}{n}\right) =
   1-\dfrac{3}{4}\left(1-\dfrac{|V_2|}{n}\right)\geq 1-\dfrac{3}{4}\left(\dfrac{|V_1|}{n}+\dfrac{4^{t+2}d_0}{Mn}\right) \geq 1 - \frac {|V_1|}{n},$$
   where the final inequality holds since $\dfrac{4^{t+2}d_0}{Mn} < \dfrac{|V_1|}{3n}$. Therefore, we get, as desired, that
       \begin{align*}
            |U_1| + |U_2| 
                \geq 
            \dfrac{k_1}{g_1(k_1)} + \dfrac{k_2}{g_1(k_2)} 
                &=
            \dfrac{k}{g_1(k_1)}\cdot\left(\dfrac{|V_1|}{n}\right)^{\frac{2}{3}}+\dfrac{k}{g_1(k_2)}\cdot\left(\dfrac{|V_2|}{n}\right)^{\frac{2}{3}}\\
                & \geq
            \dfrac{k}{g_1(k)}\cdot\dfrac{|V_1|}{n}+\dfrac{k}{g_1(k)}\cdot\left(1-\dfrac{|V_1|}{n}\right)
                \geq \dfrac{k}{g_1(k)}. 
       \end{align*}
    Noting also that $\max _{i\in [2]} |U_i| \geq k/2g_1(k) \geq k^{1/2} \geq x_0$ this gives \textbf{(B)}.\vspace{3mm}

    \noindent \underline{\textbf{Case 1(b)}}: $d_G(v_0) = d_0\in (k^{3/2}/T,2n^4/k^{6})$.\vspace{3mm} 

As $d_0$ is smaller here than in Case 1(a), the bound on $|V_2|$ from Case 1(a) still applies, and in\linebreak particular $|V_2| \geq n^4/k^6$. Thus $V_2$ falls into \hyperlink{O1}{ \textbf{(O1)}}. Setting $k_2:= k(V_2)$, this gives a set $U_2 \subset V_2$ with $|U_2| \geq k/g_1(k_2)$ and a distribution ${\cal D}_2$ on $[0.1, 0.9]^{V_2}$ with $\baddd{D}{2}{}{U} \leq |U_2|\cdot g_1(|U_2|)$.

Next, we again have $d_0 / 2 \leq |V_1| \leq d_0 \leq 2n^4/k^6$. Take a set $V_1' \subset V_1$ of size $\min(|V_1|, n^4/k^6)$, so that $n^4/k^6\geq |V_1'| \geq |V_1|/2$. Therefore by \hyperlink{O2}{ \textbf{(O2)}}, taking $k_1 = k(V_1')$, there is a set $U_1 \subset V_1'$ with $|U_1| \geq k_1/g_2(k_1)$ and a distribution ${\cal D}_1$ on $[0.1,0.9]^{V_1}$ with  $\baddd{D}{1}{}{U_1} \leq |U_1|\cdot g_2(|U_1|)$.  

  As the bound $\big(|V_2|/n\big)^{2/3}\geq 1-|V_1|/n$ still applies here, it follows that
   \begin{align*}
       |U_1|+|U_2| \geq \dfrac{k_1}{g_2(k_1)} + \dfrac{k_2}{g_1(k_2)}
       & =  \dfrac{k}{g_2(k_1)}\cdot \dfrac{|V_1'|}{n} \cdot \bigg (\frac{k^2}{n} \bigg )+\dfrac{k}{g_1(k_2)}\cdot \bigg (\frac {|V_2|}{n} \bigg )^{2/3}\\
   & \geq  
   \dfrac{k}{2\cdot g_2(k)}\cdot \dfrac{|V_1|}{n} +\dfrac{k}{g_1(k)}\cdot \bigg (1 - \frac {|V_1|}{n} \bigg)\geq \dfrac{k}{g_1(k)},
   \end{align*}
   using that $g_1(k) \geq 2g_2(k)$. Noting also that $\max _{i\in [2]} |U_i| \geq k/2g_1(k) \geq k^{1/2} \geq x_0$ this gives \textbf{(B)} and completes the proof in Case 1.\vspace{3mm}
   
    \noindent \underline{\textbf{Case 2}}: $|A_2| \geq n/4$.\vspace{2mm}
   
    In this case we apply Lemma \ref{lem:cluster-partition-lemma}, taking $A_2$ as $A$, $\alpha := 4k/n$ and  $\Theta = (M, \lambda )$. To check that all the conditions apply, we note that $\log ^2n \leq n^{1/2} \leq k \leq n^{2/3} \leq n / \log ^3 n$ and that $\alpha = (4k/n) \leq 4n^{-1/3} \leq (\lambda \log ^5n)^{-1}$. Moreover, the set $A_2$ satisfies
        \begin{align*}
            A_2 
                \subset 
            \big \{v \in V(G): M \log ^2 n \leq T^{-1}\cdot k^{3/2}\leq d_G(v) \leq n/2 \mbox{ and } |W_*(v)| \leq \alpha n \big \}
                \subset 
            V(G).
        \end{align*}
    Thus the conditions of Lemma \ref{lem:cluster-partition-lemma} hold. Let $U = \{u_1,\ldots, u_t\} \subset A_2$ and $V_1,\ldots, V_t, S\subset V(G)$ be sets and let $\{d_i\}_{i\in [t]}$ be the values given by the Lemma, recalling that (\emph{ii}) gives us $t \leq k$.\vspace{2mm}

    We now show the \hyperlink{strategy}{\textbf{Strategy}}, applied to some of the $\{V_i\}_{i\in [t]}$, holds. To see that \textbf{(A)} holds, we take ${\cal E} = {\cal B}_{\beta }(U, S)$ to be a blended distribution on $[0.1, 0.9]^S$. By Lemma \ref{lem:cluster-partition-lemma} (\emph{iii}) and (\emph{iv}),\linebreak the hypothesis of Lemma \ref{lem: blended-distribution-control-sets} is true for $D = (5M)^{-1}\cdot \min _{i\in [t]} d_G(u_i) \geq (5TM)^{-1}\cdot k^{3/2}$. Since $U$ is trivially $1$-balanced to $S$, 
    we take $\gamma = 1$ and $\beta ^{-1} = 10 \sqrt{k \log k}$, so that for all $v_i \in V_i$ and $v_j \in V_j$ with $i \neq j$ we have the following `bad' control: 
    \begin{align}\label{eqn:bad-control-for-larger-degree}
        \baddd {\cal E}{}{S}{v_i,v_j} 
            & \leq 
        \dfrac{2}{\beta D} 
            + 
        2 \cdot \max \big \{d^S_G(u_i),d^S_G(u_{j}) \big \}\cdot \exp\left(\dfrac{-0.045}{\gamma\beta^2|U|}\right) \nonumber \\
            & \leq 
        \frac{10TM\cdot 10\sqrt{k\log k} }{k^{3/2}}+2\cdot \dfrac{k^2}{2} \cdot \exp\left(\frac{-4.5\cdot k\log k}{t}\right)\nonumber \\
            & \leq 
        \frac{100TM\sqrt{\log k}}{k}+\frac{1}{k^2}\leq \frac{200 TM\sqrt{\log (k)}}{k} \nonumber \\ 
            &\leq 
        \frac {\exp \Big (\frac{C_2}{2}(\log k)^{2/3} \Big )}{k}\leq g_1'(2k).
    \end{align}
    This confirms \textbf{(A)} from the \hyperlink{strategy}{\textbf{Strategy}}, provided that $C_2$ is large enough.\vspace{3mm}
  
    \noindent \underline{\textbf{Case 2(a)}}: $k\geq |U| = t \geq k / \lambda ^3$.\vspace{3mm}
        
    Here we can proceed directly using \eqref{eqn:bad-control-for-larger-degree}. We have $U \subset V(G)$ with $|U| = k/\lambda ^3 >k/g_1(k)$ and a distribution ${\cal D} = {\cal E} \times {\cal T}$ on $[0.1, 0.9]^{V(G)}$, where ${\cal T}$ is the trivial distribution on $[0,1,0.9]^{V(G) \setminus S}$, which by \eqref{eqn:bad-control-for-larger-degree} satisfies $\bad {U} \leq \binom {|U|}{2} \cdot 200\cdot TM \cdot ({\sqrt{\log (k)}}/{k}) \leq |U|\cdot g_1(|U|)$, using $|U| \leq k$, as required.\vspace{3mm}
  
    \noindent \underline{\textbf{Case 2(b)}}: $|U| = t < k/\lambda ^3$.\vspace{3mm}
    
    Here we continue with the \hyperlink{strategy}{\textbf{Strategy}}, showing that \textbf{(B)} also holds. Consider the partition $[t] = I_1 \cup I_2$, where $I_1 = \{i\in [t]: |V_i| \geq n^4/k^6\}$ and $I_2 = \{i\in [t]: |V_i| < n^4/k^6\}$. As $t < k$, by the `Furthermore' statement in Lemma \ref{lem:cluster-partition-lemma} (\emph{ii}) we have
        $$\sum _{i\in I_1}|V_i|+ \sum _{i \in I_2} |V_i| \geq \frac {n}{1000\lambda \log ^2n} \geq 2\bigg ( \frac{n}{\lambda ^2}\bigg ).$$
    It follows from above that for some $j \in \{1,2\}$ we have 
        $\sum _{i \in I_j} |V_i| \geq n/{\lambda ^2}$. As $|I_j| \leq t \leq k/\lambda ^3$, this gives $\max _{i\in I_j} |V_i| \geq \lambda (n/k)$. 
    Let us also observe\footnote{While $|V_i| \leq n/\lambda ^6$ is a much weaker than our known bound of $|V_i| \leq 4k$, our current analysis will also apply Case 3 below, where only a weaker bound can be achieved.} that  $|V_i| \leq |W_*(u_i)| \leq 4k \leq n/\lambda ^6$.\vspace{3mm}

    \noindent \hypertarget{size-control-claim}{\textbf{\underline {Claim:}}} Suppose we have sets $\{V_i\}_{i\in I_j}$ with $\sum _{i\in I_j} |V_i| \geq n/\lambda ^2$ and $\lambda (n/k) \leq \max_{i\in I_j} |V_i| \leq n/\lambda ^6$. Then the sets $\{U_i\}_{i\in I_j}$ given by the \hyperlink{strategy}{\textbf{Strategy}} satisfy \textbf{(B)}, i.e.
    $$\sum _{i\in I_j}|U_i| \geq k/g_1(k)\ \  \ \text{and}\ \ \ \max _{i\in I_j}|U_i| \geq x_0.$$
    
    To see this, write $k_i:=k(V_i)$ for all $i\in I_j$ and let $U_0=\bigcup_{i\in I_j}U_i$. Then $|U_0|=\sum_{i\in I_j}|U_i|$.
    If $j = 1$ then the sets $V_i$ fall into \hyperlink{O1}{ \textbf{(O1)}} for $i \in J_1$. Then $k_i = k\cdot (|V_i|/n)^{2/3}$ and hence
        \begin{align*}
           |U_0|\geq \sum _{i\in I_1} \dfrac{k_i}{g_1(k_i)}
                & =  
            \sum _{i\in I_1} \dfrac{k}{g_1(k_i)} \cdot \bigg ( \frac{|V_i|}{n} \bigg )^{2/3}
                 = 
            \dfrac{k}{g_1(k)} \cdot \sum _{i\in I_1}  \frac{|V_i|}{n}  \cdot \bigg ( \frac{n}{|V_i|} \bigg )^{1/3} \\
                 & \geq 
            \dfrac{k}{g_1(k)} \cdot \bigg ( \sum _{i\in I_1}  \frac{|V_i|}{n} \bigg ) \cdot \bigg ( \frac{n}{n/\lambda ^6} \bigg )^{1/3}     \geq 
            \dfrac{k}{g_1(k)} \cdot \frac{1}{\lambda ^2} \cdot \lambda ^2 = \dfrac{k}{g_1(k)}. 
        \end{align*}
    Besides, taking $i\in I_j$ such that $|V_i| \geq \lambda (n/k)$, we have $k_i \geq k (|V_i|/n)^{2/3} 
    \geq k^{1/3}$. Therefore $|U_i| \geq k_i/g_1(k_i) \geq k^{1/6} \geq x_0$. 
    
    If $j = 2$ instead then the sets $V_i$ fall into \hyperlink{O2}{ \textbf{(O2)}} for $i \in J_2$. Then $k_i = |V_i| (k^3/n^2)$ and so 
        \begin{align*}
            |U_0|\geq \sum _{i\in I_2} \dfrac{k_i}{g_2(k_i)} = \sum _{i\in I_2} \dfrac{k}{g_2(k_i)} \cdot \bigg (\frac {|V_i|}{n} \bigg ) \cdot \bigg ( \frac{k^2}{n} \bigg ) \geq \dfrac{k}{g_2(k)} \cdot \bigg ( \sum _{i\in I_2} \frac{|V_i|}{n} \bigg ) \cdot 1 
                \geq \frac {k}{\lambda ^2\cdot g_2(k)}.
        \end{align*}
    Also, taking $i\in I_j$ such that $|V_i| \geq \lambda (n/k)$, we have $k_i \geq |V_i|(k^3/n^2) 
    \geq \lambda (k^2/n) \geq \lambda $. Therefore $|U_i| \geq k_i/g_1(k_i) \geq \lambda ^{1/2} \geq \log k \geq x_0$. 
        
   \indent By combining these two bounds we obtain, as required, that
        \begin{align*}
            \sum _{i\in I_j} |U_i| \geq \min \bigg ( \frac{k}{g_1(k)}, \frac{k}{\lambda ^2 \cdot g_2(k)} \bigg ) = \frac{k}{g_1(k)}.
        \end{align*}
    Therefore \textbf{(B)} holds for $\{U_i\}_{i\in I_j}$, completing the proof of the \hyperlink{size-control-claim}{\textbf{{Claim}}}. Since \textbf{(A)} also holds, this means that we can apply the \hyperlink{strategy}{\textbf{Strategy}} for the set $U_0$, thus completing Case 2. \vspace{3mm}

\noindent \underline{\textbf{Case 3}}: $|A_3| \geq n/4$.\vspace{2mm}

This case has some similarities with the approach from Case 2 above, though there are also several new elements and the analysis is more subtle at certain points. 

To begin, note that $G[A_3]$ is an induced subgraph of $G$ with $|G[A_3]| \geq n/4$ and, by definition of $A_3$, we have $\Delta (G[A_3]) \leq k^{3/2}/T$. We apply Theorem \ref{AKS-type-thm} to $G[A_3]$ to find an induced subgraph $H$ of $G[A_3]$ such that $m:= |H| \geq |A_3|/30\log n \geq n/120\log n$ with the property that there is $D$ such that $D \leq d_G(v) \leq 5 D \log n$ for all $v \in V(H)$. Since $\hom (H) \leq \hom (G) \leq n^2/k^3$, we note that we must have $D \geq |H|\cdot (n^2/k^3 + 1)^{-1} \geq k^3\cdot (250n \log n)^{-1} \geq k^{1/2}$ by Tur\'an's theorem. Clearly also $D \leq \Delta (H) \leq T^{-1}\cdot k^{3/2} < m/2$. 

We now set $\Theta = (M, \lambda)$ and consider the $\Theta$-clusters, \emph{this time in} $H$. We first show that  Lemma \ref{lem:cluster-partition-lemma} applies for $A = V(H)$ and $\alpha := (Tm)^{-1}\cdot 2k^{3/2}$. We have $\log^2m\leq k \leq m\log^{-3}m$\linebreak since $|A|=m\geq 2n/\log^2 n$, and  
 $\alpha = 2k^{3/2}/(Tm) \leq k^{3/2}\log^2n/(Tn) \leq \log ^2n/T \leq (\lambda \log ^5m)^{-1}$. 
 Given $v \in V(H)$, using Lemma \ref{lem:simple-cluster-properties} (iii) we have $T_{\Theta }^{V(H)}(v) \leq \log _{\lambda }(n) < \log _4M - 1$, and so by Lemma \ref{lem:cluster-bound} we have $|W_*(v)| \leq 2\cdot\Delta (H) \leq 2k^{3/2}/T < \alpha m$. To check the remaining condition, observe that 
 $M \log ^2m \leq k^{1/2} \leq D \leq d_H(v) \leq \Delta (H) < m/2$, hence: 
    \begin{align*}
        A 
            \subset  
        \Big \{v \in V(H): M \log ^2 m \leq d_H(v) \leq m/2 \mbox{ and } \big |W_*(v) \big | \leq \alpha m \Big \}
            \subset 
        V(H).
    \end{align*}
\indent Thus Lemma \ref{lem:cluster-partition-lemma} applies with these parameters to give sets $U,V_1,\ldots, V_t, S$ and values $\{d_i\}_{i\in [t]}$ which satisfy conditions (\emph{i})-(\emph{v}) (with $n$ replaced by $m$ in these statements). As advised by (\emph{v}),\linebreak we also let $\gamma := \log ^5m \cdot \max (\Delta (G)\cdot m^{-1}, t^{-1}) \leq \log ^8n \cdot \max (D/n, 1/t)$.

We once again show the \hyperlink{strategy}{\textbf{Strategy}} holds when applied to some of the $\{V_i\}_{i\in [t]}$. To see that \textbf{(A)} holds, recall from Lemma \ref{lem:cluster-partition-lemma} (\emph{v}) that $U$ is $\gamma $-balanced to $S$. Set $\beta ^{-1} = 10(\gamma |U| \log ^4k)^{1/2}$ and take ${\cal E} = {\cal B}_{\beta }(U, S)$ to be the blended distribution on $[0.1, 0.9]^S$. By recalling the definition of the $\Theta$-cluster, Lemma \ref{lem: blended-distribution-control-sets} for $D/M$ gives that for all $v_i \in V_i$ and $v_j \in V_j$ with $i \neq j$ we have
    \begin{align} \label{eqn:bad-control-case-3}
        \baddd {\cal E}{}{S}{v_i,v_j} 
            & \leq 
        \frac{2M}{\beta D} + 2 \cdot (5D \log n) \cdot \exp \bigg (\frac{-0.045}{\gamma \beta ^2 |U|} \bigg )\nonumber \\
            & \leq 
        \frac {20\cdot M \sqrt{\gamma |U| \log ^4 k} }{D} + \dfrac{10\cdot k^{3/2}\log n}{T}\cdot \exp \big( - 4.5\cdot \log ^4k \big)\nonumber \\
            & \leq 
        \dfrac{20M\log^6k}{D}\cdot { \max \bigg ( \sqrt{\frac{D|U|}{n}}, 1 \bigg ) } + k^{-3} 
            \leq 
        \frac{42M\log ^{7}k }{k} \leq g_1'(2k).
    \end{align}
    The final inequality above holds since $D \geq k^3/ (n\log ^{2}n)\geq k^3/(4n\log^{2}k$), which implies that\linebreak 
     $\sqrt{|U|/Dn} \leq \sqrt{4\log^2k\cdot k^{-2}}=2\log k/k$.\ Thus 
     part \textbf{(A)} of the \hyperlink{strategy}{\textbf{Strategy}} holds.\vspace{2mm}

    \noindent \underline{\textbf{Case 3(a)}}: $k\geq |U| = t \geq k / \lambda ^3$.\vspace{2mm}
        
    As before, in this case we proceed directly since 
    $|U| = k/\lambda ^3 \geq k/g_1(k)$ and there is already a distribution ${\cal D} = {\cal E} \times {\cal T}$ on $[0.1, 0.9]^{V(G)}$, where ${\cal T}$ is the trivial distribution on $[0,1,0.9]^{V(G) \setminus S}$. From \eqref{eqn:bad-control-case-3} we deduce that $\bad {U} \leq \binom {|U|}{2} \cdot 42Mk^{-1} \log^7 k \leq |U|\cdot g_1(|U|)$, using $|U| = t \leq k$, as required. 
  
    \noindent \underline{\textbf{Case 3(b)}}: $|U| = t < k/ \lambda ^3$.\vspace{2mm}
    
    Here we continue with the \hyperlink{strategy}{\textbf{Strategy}}, verifying that \textbf{(B)} holds. Let $[t] = I_1 \cup I_2$ denote again the partition with $I_1 = \{i\in [t]: |V_i| \geq n^4/k^6\}$ and $I_2 = \{i\in [t]: |V_i| < n^4/k^6\}$. As $t < k$, by Lemma \ref{lem:cluster-partition-lemma} (\emph{ii}) we obtain that
        $$\sum _{i\in I_{1}}|V_i|+ \sum _{i \in I_2} |V_i| \geq \frac {m}{10^3\lambda \log ^2m} \geq \frac{n}{10^3 \lambda \log ^4n } \geq 2\bigg ( \frac{n}{\lambda ^2}\bigg ).$$
    By using the pigeonhole principle, there is $j \in \{1,2\}$ such that 
        $$\sum _{i \in I_j} |V_i| \geq \dfrac{n}{\lambda ^2}.$$
    Moreover, $\lambda (n/k) \leq \sum _{i\in I_j} |V_i| / t \leq  \max _{i\in I_j} |V_i| \leq \alpha n \leq k^{3/2}\log ^2n/T \leq n/ \lambda ^6$. It follows that the \hyperlink{size-control-claim}{\textbf{{Claim}}} from Case 2 above can be applied to the current sets $\{V_i\}_{i\in I_j}$ as well, so \textbf{(B)} holds. Therefore the \hyperlink{strategy}{\textbf{Strategy}} once again applies, completing the proof of Case 3.\vspace{2mm}
    
    We finish the proof by noting that our three cases exhaust all possibilities, since otherwise  $|A| = |A_1| + |A_2| + |A_3| < 0 + (n/4) + (n/4) = n/2$, contrary to the assumption that $|A| \geq n/2$. Therefore one of these cases must apply and so the theorem is proved. \qed



\printbibliography

@article{LongPL,
title = {Distinct degrees and homogeneous sets},
journal = {Journal of Combinatorial Theory, Series B},
volume = {159},
pages = {61-100},
year = {2023},
issn = {0095-8956},
doi = {https://doi.org/10.1016/j.jctb.2022.11.004},
url = {https://www.sciencedirect.com/science/article/pii/S0095895622001216},
author = {Eoin Long and Laurenţiu Ploscaru},
}

@article{quarter-binomial,
title = {Tight lower bound on the probability of a binomial exceeding its expectation},
journal = {Statistics \& Probability Letters},
volume = {86},
pages = {91-98},
year = {2014},
issn = {0167-7152},
doi = {https://doi.org/10.1016/j.spl.2013.12.009},
url = {https://www.sciencedirect.com/science/article/pii/S0167715213004082},
author = {Spencer Greenberg and Mehryar Mohri}}

@article{bukh,
title = {Induced subgraphs of Ramsey graphs with many distinct degrees},
journal = {Journal of Combinatorial Theory, Series B},
volume = {97},
number = {4},
pages = {612-619},
year = {2007},
issn = {0095-8956},
doi = {https://doi.org/10.1016/j.jctb.2006.09.006},
url = {https://www.sciencedirect.com/science/article/pii/S0095895606001080},
author = {Boris Bukh and Benny Sudakov}
}

@article{JKLY,
title = "Distinct degrees in induced subgraphs",
author = "Matthew Jenssen and Eoin Long and Peter Keevash and Liana Yepremyan",
year = "2020",
month = sep,
doi = "10.1090/proc/15060",
volume = "148",
pages = "3835--3846",
journal = "Proceedings of the American Mathematical Society",
issn = "0002-9939",
publisher = "American Mathematical Society",
number = "9",
}

@Inbook{Erdosold,
author="Erd\H{o}s, P. and Szekeres, G.",
editor="Gessel, Ira
and Rota, Gian-Carlo",
title="A Combinatorial Problem in Geometry",
bookTitle="Classic Papers in Combinatorics",
year="1987",
publisher="Birkh{\"a}user Boston",
address="Boston, MA",
pages="49--56",
isbn="978-0-8176-4842-8",
doi="10.1007/978-0-8176-4842-8_3",
url="https://doi.org/10.1007/978-0-8176-4842-8_3"
}

@article{Erdosprob,
author = {P. Erd\H{o}s}, 
title = {{Some remarks on the theory of graphs}},
volume = {53},
journal = {Bulletin of the American Mathematical Society},
number = {4},
publisher = {American Mathematical Society},
pages = {292 -- 294},
year = {1947},
doi = {bams/1183510596},
}

@book{bol-rg, place={Cambridge}, edition={2}, series={Cambridge Studies in Advanced Mathematics}, title={Random Graphs}, DOI={10.1017/CBO9780511814068}, publisher={Cambridge University Press}, author={Bollobás, Béla}, year={2001}, collection={Cambridge Studies in Advanced Mathematics}}

@article{Pyber199541,
title = {Dense Graphs without 3-Regular Subgraphs},
journal = {Journal of Combinatorial Theory, Series B},
volume = {63},
number = {1},
pages = {41-54},
year = {1995},
issn = {0095-8956},
doi = {https://doi.org/10.1006/jctb.1995.1004},
url = {https://www.sciencedirect.com/science/article/pii/S0095895685710040},
author = {L. Pyber and V. R\"odl and E. Szemer\'edi}
}

@article{ramsey1930,
author = {Ramsey, F. P.},
title = {On a Problem of Formal Logic},
journal = {Proceedings of the London Mathematical Society},
volume = {s2-30},
number = {1},
pages = {264-286},
doi = {https://doi.org/10.1112/plms/s2-30.1.264},
url = {https://londmathsoc.onlinelibrary.wiley.com/doi/abs/10.1112/plms/s2-30.1.264},
year = {1930}
}

@article{longetcomp,
title = "On a Ramsey-type problem of Erd{\H o}s and Pach",
author = "Ross Kang and Eoin Long and Viresh Patel and Guus Regts",
year = "2017",
month = dec,
day = "1",
doi = "10.1112/blms.12094",
language = "English",
volume = "49",
pages = "991--999",
journal = "Bulletin of the London Mathematical Society",
issn = "0024-6093",
publisher = "London Mathematical Society",
number = "6",
}

@article{Alon198479,
title = {Regular subgraphs of almost regular graphs},
journal = {Journal of Combinatorial Theory, Series B},
volume = {37},
number = {1},
pages = {79-91},
year = {1984},
issn = {0095-8956},
doi = {https://doi.org/10.1016/0095-8956(84)90047-9},
url = {https://www.sciencedirect.com/science/article/pii/0095895684900479},
author = {N Alon and S Friedland and G Kalai},
}

@book{lovasz1993combinatorial,
  title={Combinatorial Problems and Exercises},
  author={Lov{\'a}sz, L.},
  isbn={9780821869475},
  series={AMS/Chelsea publication},
  url={https://books.google.co.uk/books?id=e99fXXYx9zcC},
  year={1993},
  publisher={North-Holland Publishing Company}
}

@inbook{ramsey-survey, place={Cambridge}, series={London Mathematical Society Lecture Note Series}, title={Recent developments in graph Ramsey theory}, DOI={10.1017/CBO9781316106853.003}, booktitle={Surveys in Combinatorics 2015}, publisher={Cambridge University Press}, author={Conlon, David and Fox, Jacob and Sudakov, Benny}, editor={Czumaj, Artur and Georgakopoulos, Agelos and Král, Daniel and Lozin, Vadim and Pikhurko, Oleg (Editors)}, year={2015}, pages={49–118}, collection={London Mathematical Society Lecture Note Series}}

@article{erdosmrd,
title = {On a Ramsey-type theorem},
journal = {Period Math Hung},
volume = {2},
number = {1},
pages = {295–299},
year = {1972},
doi = {https://doi.org/10.1007/BF02018669},
url = {https://link.springer.com/article/10.1007\%2FBF02018669#citeas},
author = {P. Erd\H{o}s and A. Szemer\'edi}
}

@article{PROMEL,
title = {Non-Ramsey Graphs Are clogn-Universal},
journal = {Journal of Combinatorial Theory, Series A},
volume = {88},
number = {2},
pages = {379-384},
year = {1999},
issn = {0097-3165},
doi = {https://doi.org/10.1006/jcta.1999.2972},
url = {https://www.sciencedirect.com/science/article/pii/S0097316599929722},
author = {H. J. Pr\"omel and V. R\"odl}}

@article{Shelah,
  title={Erd\H{o}s and R{\'e}nyi Conjecture},
  author={S. Shelah},
  journal={J. Comb. Theory, Ser. A},
  year={1998},
  volume={82},
  pages={179-185}
}

@article{narayanan2017ramsey,
author = {Narayanan, Bhargav and Sahasrabudhe, Julian and Tomon, Istv\'an},
year = {2016},
month = {09},
pages = {},
title = {Ramsey Graphs Induce Subgraphs of Many Different Sizes},
volume = {39},
journal = {Combinatorica},
doi = {10.1007/s00493-017-3755-0}
}

@article{KSP,
author = {Kwan, Matthew and Sudakov, Benny},
year = {2017},
month = {12},
pages = {},
title = {Proof of a conjecture on induced subgraphs of Ramsey graphs},
volume = {372},
journal = {Transactions of the American Mathematical Society},
doi = {10.1090/tran/7729}
}

@article{sos,
title = {Some of my favourite problems in various branches of Combinatorics },
journal = {Le mathematiche},
volume = {47},
number = {2},
pages = {231-240},
year = {1992},
url = {https://lematematiche.dmi.unict.it/index.php/lematematiche/article/view/587},
author = {P. Erd\H{o}s}}

@article{unpublished,
title = {The number of distinct degrees in an induced subgraph of a random graph },
journal = {unpublished},
author = {D. Conlon and R. Morris and W. Samotij and D. Saxton}}

@misc{ferber2021graph,
      title={Every graph contains a linearly sized induced subgraph with all degrees odd}, 
      author={Asaf Ferber and Michael Krivelevich},
      year={2021},
      eprint={2009.05495},
      archivePrefix={arXiv},
      primaryClass={math.CO}
}

@article{scott_1992, title={Large Induced Subgraphs with All Degrees Odd}, volume={1}, DOI={10.1017/S0963548300000389}, number={4}, journal={Combinatorics, Probability and Computing}, publisher={Cambridge University Press}, author={Scott, A. D.}, year={1992}, pages={335–349}}

@article{narayanan, title={Induced Subgraphs With Many Distinct Degrees}, volume={27}, DOI={10.1017/S0963548317000256}, number={1}, journal={Combinatorics, Probability and Computing}, publisher={Cambridge University Press}, author={Narayanan, Bhargav and Tomon, Istv\'{a}n}, year={2018}, pages={110–123}}

@article{alonsud,
author = {Alon, Noga and Krivelevich, Michael and Sudakov, Benny},
title = {Induced subgraphs of prescribed size},
journal = {Journal of Graph Theory},
volume = {43},
number = {4},
pages = {239-251},
doi = {https://doi.org/10.1002/jgt.10117},
url = {https://onlinelibrary.wiley.com/doi/abs/10.1002/jgt.10117},
eprint = {https://onlinelibrary.wiley.com/doi/pdf/10.1002/jgt.10117},
year = {2003}
}

@book{bollo,
author = {Bollob\'{a}s, B.},
title = {Extremal Graph Theory},
year = {2004},
isbn = {0486435962},
publisher = {Dover Publications},
address = {USA}
}

@book{spencert,
  title={Ramsey Theory},
  author={Graham, R.L. and Rothschild, B.L. and Spencer, J.H.},
  isbn={9780471500469},
  lccn={89022670},
  series={Wiley Series in Discrete Mathematics and Optimization},
  year={1991},
  publisher={Wiley}
}

@book{ramírez2001perfect,
  title={Perfect Graphs},
  author={Ram{\'i}rez-Alfons{\'i}n, J.L. and Reed, B.A.},
  isbn={9780471489702},
  lccn={20010225},
  series={Wiley Series in Discrete Mathematics and Optimization},
  year={2001},
  publisher={Wiley}
}

@book{alon04,
  address = {New York},
  author = {Alon, Noga and Spencer, Joel H.},
  edition = {Second},
  isbn = {0471370460},
  publisher = {Wiley},
  refid = {85820345},
  title = {The Probabilistic Method},
  year = 2004
}

@book{book-erdos,
author = {Chung, Fan and Graham, Ron},
year = {2020},
publisher = {A K Peters, Ltd.},
month = {07},
pages = {},
title = {Erdős on Graphs: His Legacy of Unsolved Problems},
isbn = {9781003075950},
doi = {10.1201/9781003075950}
}

@article{barak20122,
  title={2-source dispersers for $n^{o(1)}$-entropy, and Ramsey graphs beating the Frankl-Wilson construction},
  author={Barak, Boaz and Rao, Anup and Shaltiel, Ronen and Wigderson, Avi},
  journal={Annals of Mathematics},
  pages={1483--1543},
  year={2012},
  publisher={JSTOR}
}

@inproceedings{xinli,
author = {Li, Xin},
title = {Improved Non-Malleable Extractors, Non-Malleable Codes and Independent Source Extractors},
year = {2017},
isbn = {9781450345286},
publisher = {Association for Computing Machinery},
address = {New York, NY, USA},
url = {https://doi.org/10.1145/3055399.3055486},
doi = {10.1145/3055399.3055486},
booktitle = {Proceedings of the 49th Annual ACM SIGACT Symposium on Theory of Computing},
pages = {1144–1156},
numpages = {13},
location = {Montreal, Canada},
series = {STOC 2017}
}

@article{LO,
author = {P. Erd\H{o}s},
title = {{On a lemma of Littlewood and Offord}},
volume = {51},
journal = {Bulletin of the American Mathematical Society},
number = {12},
publisher = {American Mathematical Society},
pages = {898 -- 902},
year = {1945},
doi = {bams/1183507531},
URL = {https://doi.org/}
}

@article{kwan2018ramsey,
author = {Kwan, Matthew and Sudakov, Benny},
year = {2017},
month = {11},
pages = {},
title = {Ramsey Graphs Induce Subgraphs of Quadratically Many Sizes},
volume = {2020},
journal = {International Mathematics Research Notices},
doi = {10.1093/imrn/rny064}
}

@article{cliques, title={Cliques in random graphs}, volume={80}, DOI={10.1017/S0305004100053056}, number={3}, journal={Mathematical Proceedings of the Cambridge Philosophical Society}, publisher={Cambridge University Press}, author={Bollob\'{a}s, B\'{e}la and Erd\H{o}s, Paul}, year={1976}, pages={419–- 427}}

@article{alon2008large,
title = "Large nearly regular induced subgraphs",
author = "Noga Alon and Michael Krivelevich and Benny Sudakov",
year = "2008",
doi = "10.1137/070704927",
language = "English (US)",
volume = "22",
pages = "1325--1337",
journal = "SIAM Journal on Discrete Mathematics",
issn = "0895-4801",
publisher = "Society for Industrial and Applied Mathematics Publications",
number = "4",
}

@article{KSSS, 
title={Anticoncentration in Ramsey graphs and a proof of the Erdős–McKay conjecture}, 
volume={11}, DOI={10.1017/fmp.2023.17}, 
journal={Forum of Mathematics, Pi}, 
author={Kwan, Matthew and Sah, Ashwin and Sauermann, Lisa and Sawhney, Mehtaab}, 
year={2023}, pages={e21}}

@article{quasi,
author = {Simonovits, Mikl\'os and S\'os, Vera T.},
title = {Szemer\'edi's partition and quasirandomness},
journal = {Random Structures \& Algorithms},
volume = {2},
number = {1},
pages = {1-10},
doi = {https://doi.org/10.1002/rsa.3240020102},
url = {https://onlinelibrary.wiley.com/doi/abs/10.1002/rsa.3240020102},
year = {1991}
}

@incollection{THOMASON-quasi,
title = {Pseudo-Random Graphs},
series = {North-Holland Mathematics Studies},
publisher = {North-Holland},
volume = {144},
pages = {307-331},
year = {1987},
booktitle = {Annals of Discrete Mathematics (33)},
issn = {0304-0208},
doi = {https://doi.org/10.1016/S0304-0208(08)73063-9},
url = {https://www.sciencedirect.com/science/article/pii/S0304020808730639},
author = {Andrew Thomason},
}

@article{quasi-2,
  author = {Chung, F. and Graham, R. and Wilson, R.},
  title = {Quasi-random graphs},
  journal = {Proceedings of the National Academy of Sciences},
  year = {1988},
  volume = {85},
  issue = {4},
  pages = {969-970},
  doi = {10.1073/pnas.85.4.969}
}

\end{document}